\documentclass[fleqn,11pt]{article}
\usepackage{amsmath,amssymb,amsthm,xcolor,esint}
\usepackage[english]{babel}
\usepackage[margin=3cm]{geometry}
\usepackage{csquotes}
\usepackage[bookmarks]{hyperref}
\usepackage{graphicx}
\usepackage{authblk}

\usepackage[title]{appendix}

\newcommand{\R}{{\mathbb R}}

\newcommand{\Z}{{\mathbb Z}}

\newcommand{\e}{\varepsilon}
\DeclareMathOperator{\dv}{div}

\DeclareMathOperator{\dist}{dist}
\DeclareMathOperator{\diam}{diam}

\newcommand{\BV}{{\mathrm{BV}}}
\newcommand{\TV}{{\mathrm{TV}}}

\renewcommand{\L}{{\mathcal{L}}}
\newcommand{\p}{g}

\newtheorem{thm}{Theorem}[section]
\newtheorem{prop}[thm]{Proposition}
\newtheorem{lem}[thm]{Lemma}
\newtheorem{cor}[thm]{Corollary}
\newtheorem{definition}[thm]{Definition}
\theoremstyle{definition}
\newtheorem{rem}[thm]{Remark}

\title{Stability of the vortex in micromagnetics and related models}
\date{}

\begin{document}

\author{Xavier Lamy\thanks{Institut de Math\'ematiques de Toulouse; UMR 5219, Universit\'e de Toulouse; CNRS, UPS IMT, F-31062 Toulouse Cedex 9, France. Email: xlamy@math.univ-toulouse.fr.}
 \qquad  Elio Marconi\thanks{Dipartimento di Matematica 'Tullio Levi Civita', Universit\`a di Padova, via Trieste 63, 35121 Padova (PD), Italy. Email: elio.marconi@unipd.it.}}

\maketitle

\abstract{
We consider  line-energy models 
 of Ginzburg-Landau type in a 
 two-dimensional 
 simply-connected bounded domain. 
Configurations of vanishing energy 
have been characterized by Jabin, Otto and Perthame: 
the domain must be a disk, and the configuration a vortex.
We prove  a quantitative version of this statement in the class of $C^{1,1}$ domains,
improving on previous results by Lorent.
In particular,
the deviation of the domain from a disk is controlled
 by a power of the energy, and that power is optimal.
The main tool is a Lagrangian representation introduced by the second author, 
which allows to decompose the energy along characteristic curves.}

\section{Introduction}

\subsection{Models}

Several models arising in a variety of physical applications (micromagnetics, smectic liquid crystals, blistering) have in common that, as a characteristics length scale $\e$ tends to 0, bounded-energy configurations converge to two-dimensional vector fields $m\colon\Omega\to\R^2$ satisfying the eikonal equation
\begin{align}\label{eq:eik}
 |m|=1\text{ a.e. in }\Omega,\qquad \nabla\cdot (\mathbf 1_\Omega m)=0\text{ in }\R^2.
\end{align}
Here $\Omega\subset\R^2$ is a smooth, simply connected bounded domain, and the divergence constraint on the trivially extended field $\mathbf 1_\Omega m$ amounts to $\nabla\cdot m=0$ in $\Omega$ and $m\cdot n_{\partial\Omega}=0$ on $\partial\Omega$, where  $n_{\partial\Omega}$ is the exterior unit normal (and the last condition makes sense whenever  $m$ admits a strong trace on $\partial\Omega$). Examples of such models include:
\begin{itemize} 
\item The Aviles-Giga functional, introduced in \cite{AG86} as a simplified model for smectic liquid crystals and proposed as a model for thin film blisters in \cite{gioort} (see the introduction of \cite{JK00} for other applications),
\begin{align}\label{eq:AG}
&E^{AG}_{\e}(m;\Omega)=\frac \e 2 \int_{\Omega}   |\nabla m|^2+ \frac 1{2\e}\int_\Omega (1-|m|^2)^2,\\
& m\colon\Omega\to \R^2,\qquad \nabla\cdot(\mathbf 1_\Omega m )=0\text{ in }\R^2.
\nonumber
\end{align}
Note that the Aviles-Giga functional is more often expressed in terms of $u$ such that $\nabla^\perp u=m$ in $\Omega$ and $ u=0$ on $\partial\Omega$, however in a simply connected domain the two formulations are equivalent.
\item A micromagnetics model studied in \cite{RS01,RS03},
\begin{align}\label{eq:RS}
&E^{RS}_\e(m;\Omega)=\frac \e 2 \int_{\Omega}   |\nabla m|^2
+\frac 1{2\e} \int_{\R^2} |H|^2,\\
& m\colon\Omega\to\mathbb S^1\subset\R^2,\qquad H\colon\R^2\to\R^2,
\nonumber
\\
& \nabla\times H=0\text{ and }\nabla\cdot (H +  1_{\Omega}m)=0\text{ in }\R^2.
\nonumber
\end{align}
\item  A more general micromagnetics model studied in \cite{ARS},
\begin{align}\label{eq:ARS}
&E^{ARS}_\e(m;\Omega)=\frac{\e}{2}\int_\Omega |\nabla m|^2 + \frac{1}{2\e}\int_{\R^2}|H|^2 +\frac{1}{2 c_\e}\int_\Omega |m_3|^2,\\
&m\colon \Omega\to\mathbb S^2\subset\R^3,\quad 0<c_\e \leq \e^{1+\delta},
\nonumber
\\
&H\colon\R^2\to\R^2,\qquad \nabla\times H=0\text{ and } \nabla\cdot \left(H +(m_1,m_2)\mathbf 1_\Omega\right)=0\text{ in }\R^2.
\nonumber
\end{align}
\end{itemize}
For all these models, sequences of bounded energy as $\e\to 0$ are precompact in $L^2(\Omega)$ \cite{ARS,ADM,dkmo01,RS01}, and limits of converging subsequences satisfy the eikonal equation \eqref{eq:eik}. (In the case of \eqref{eq:ARS} the limit satisfies $m_3=0$ so we can identify it with an $\R^2$-valued map.)

A large literature is devoted to understanding the behavior of minimizers $m_\e$ of $E^{AG}_\e(\cdot ; \Omega)$ as $\e \to 0$.
In particular it is conjectured in \cite{gioort} that the minimizers $m_\e$ converge to $m_\ast = \nabla^\perp\dist (\cdot, \partial \Omega)$ when $\Omega$ is convex (counterexamples in nonconvex domains are given in \cite[Theorem~7]{ignatmerlet12}). 
A positive answer is obtained in \cite{JOP} when $\Omega$ is a disk, and in \cite{marconi21ellipse} for some special domains including ellipses (under the additional boundary constraint $m_{\lfloor\partial\Omega}=-in_{\partial\Omega}$). 
For $E_\e^{RS}$ much more is known: that conjecture has been verified \cite{RS03}, and limits of non-minimizing sequences also have a well-understood structure \cite{marconi21micromag}.

In \cite{JOP} the authors characterize zero-energy states, that is, limits of sequences with energy \eqref{eq:AG} converging to 0 as $\e\to 0$. In addition to the eikonal equation \eqref{eq:eik}, zero-energy states satisfy the kinetic equation
\begin{align}\label{eq:kinzero}
e^{is}\cdot\nabla_x \mathbf 1_{m(x)\cdot e^{is}>0} =0\qquad \text{in }\Omega, \text{ for all }s\in\R.
\end{align}
This is also valid for zero-energy states of \eqref{eq:RS} \cite{JOP} and of \eqref{eq:ARS} \cite{ARS} (see Appendix~\ref{a:ent}). It is shown in \cite{JOP} that, if a smooth bounded simply connected domain $\Omega$ admits a  zero energy state, that is, a solution of \eqref{eq:eik} and \eqref{eq:kinzero}, then $\Omega$ must be a disk $\Omega=B_R(x_0)$, and $m$ must be a vortex $m(x)=\pm i (x-x_0)/|x-x_0|$, or equivalently $m=\pm\nabla^\perp \dist(\cdot,\partial \Omega)$. Various generalizations can be found in \cite{BP17,GMPS21,llp22,llp20,LP18}.

\subsection{Main results}

The main purpose of this work is to provide a quantitative version of the characterization of zero-energy states from \cite{JOP}: estimate how much $\Omega$ differs from a disk and $m$ from a vortex, in terms of the energy of an approximating sequence $m_\e\to m$. Previous results in this direction are proven in \cite{lor12,lor14}. Under the assumption that $\Omega$ is a $C^2$ convex domain renormalized to satisfy $\diam(\Omega)=2$, it is shown in \cite{lor14} that there exists $x_*\in\R^2$ such that
\begin{align}\label{eq:estimlorent}
&|\Omega\Delta B_1(x_*)| + \int_\Omega \left| m +i\frac{x-x_*}{|x-x_*|}\right|^2 dx  \leq C E^{AG}_\e(m;\Omega)^\delta,\\
&\text{whenever }\nabla\cdot m=0\text{ in }\Omega\quad\text{and }m\cdot \tau =-1\text{ on }\partial\Omega,
\nonumber
\end{align}
for some absolute constants $C>0$, $\delta=2^{-9}$, and $\tau=in_{\partial\Omega}$ a unit tangent to $\partial\Omega$. 
Note that the boundary condition $m\cdot \tau=-1$, commonly imposed in the study of the Aviles-Giga functional (see e.g. \cite{ADM,AG96,contidel,JK00}), is more restrictive than the condition $m\cdot n_{\partial\Omega}=0$ enforced in \eqref{eq:AG} (which is natural in micromagnetics models). 

Our goal is to obtain an estimate similar to \eqref{eq:estimlorent}, but with a sharp exponent $\delta$, in the limit $\e\to 0$. 
To present our results in a unified setting, we consider the energy functional
\begin{align}\label{eq:F}
F_\e(m;\Omega)&=\frac{\e}{2}\int_\Omega |\nabla m|^2 + \frac{1}{2\e}\int_{\R^2}|H|^2 
+\frac{1}{2\e}\int_{\Omega}(1-|m|^2)^2 +\frac{1}{2 \e}\int_\Omega |m_3|^4,\\
&m\colon \Omega\to \R^3,\quad H\colon\R^2\to\R^2,
\nonumber
\\
&\nabla\times H=0\text{ and } \nabla\cdot \left(H +(m_1,m_2)\mathbf 1_\Omega\right)=0\text{ in }\R^2.
\nonumber
\end{align}
This functional satisfies $F_\e \leq E_\e^{AG},E^{RS}_\e, E_\e^{ARS}$ (for $F_\e\leq E_\e^{ARS}$, note that any $m\in\mathbb S^2$ satisfies $|m_3|^4\leq |m_3|^2$) and the compactness proof of \cite{dkmo01} applies (see Appendix~\ref{a:ent}) to show that any sequence $(m_\e)$ with bounded energy $F_\e(m_\e;\Omega)\leq C$ is precompact in $L^2(\Omega)$, and its limits $m=\lim m_\e$ are $\R^2$-valued and satisfy the eikonal equation \eqref{eq:eik}.
(Similar functionals are considered in \cite{IN21}.)
We obtain a sharp bound for the $L^2$-distance  between the unit normal to $\partial\Omega$ and the unit normal to a disk, in terms of the limit of $F_\e(m_\e;\Omega)$.

\begin{thm}\label{t:main}
Let $\Omega\subset\R^2$ be a simply connected   open set of class $C^{1,1}$ with $\mathcal H^1(\partial\Omega)=2\pi$ and 
$\sup_{\partial \Omega} |\kappa| \le K$ for some $K>0$, where $\kappa$ denotes the curvature of $\partial \Omega$. There exists $c>0$ depending only on $K$ such that
\begin{align}\label{eq:main}
\inf_{x_*\in\R^2}\int_{\partial\Omega}\left| n_{\partial\Omega}(x)-\frac{x-x_*}{|x-x_*|}\right|^2 d \mathcal H^1(x)\leq c \, \liminf_{\e\to 0} \;\inf_{H^1(\Omega;\R^3)} F_\e(\cdot;\Omega),
\end{align}
where $F_\e$ is the functional defined in \eqref{eq:F}.
\end{thm}

Moreover, this estimate is sharp:

\begin{prop}\label{p:sharp}
There exist a family of convex domains $\lbrace\Omega_N\rbrace_{N\geq 3}$ of class $C^{1,1}$ with uniformly bounded curvature such that
\begin{align*}
\frac{c_1}{N^2} & \geq \inf_{x_*\in\R^2}\int_{\partial\Omega_N}\left| n_{\partial\Omega_N}(x)-\frac{x-x_*}{|x-x_*|}\right|^2 d \mathcal H^1(x) \\
& \geq c_2\, \liminf_{\e\to 0} \; \inf_{H^1(\Omega;\R^3)} F_\e(\cdot;\Omega_N) \geq \frac{c_3}{N^2},
\end{align*}
for some absolute constants $c_1,c_2,c_3>0$.
\end{prop}

\begin{rem}\label{r:sharp}
The estimate \eqref{eq:main} is sharp also when replacing $\inf F_\e$ with any of the (larger) $\inf E_\e^{AG}$, $\inf E^{RS}_\e$ or $\inf E_\e^{ARS}$, where the infimums are taken over all admissible maps for the corresponding functionals as described in \eqref{eq:AG}, \eqref{eq:RS} and \eqref{eq:ARS}. 
This will be clear from the explicit description of the $\Omega_N$'s in \S~\ref{s:sharp}.
\end{rem}

As corollaries of Theorem~\ref{t:main} and its proof we obtain two other estimates, which are however probably not sharp. The first corollary provides a bound on the distance of the boundary $\partial\Omega$ to the boundary of a disk, which is perhaps a more natural way of measuring how close $\Omega$ is to a disk.

\begin{cor}\label{c:estOmega}
Let $\Omega$ be as in Theorem~\ref{t:main}. Then 
\begin{align*}
\inf_{x_*\in\R^2}\dist(\partial\Omega,\partial B_1(x_*))\leq c \, \liminf_{\e\to 0}\; \inf F_\e(\cdot;\Omega)^{\frac 12}.
\end{align*}
for some constant $c=c(K)>0$.
\end{cor}

The second corollary provides a bound on the distance of a limiting map $m$ from a vortex.

\begin{cor}\label{c:estm}
Let $\Omega$ be as in Theorem~\ref{t:main} and $m =\lim m_\e$ as $\e\to 0$, where $(m_\e)$ is a sequence of admissible maps for the functional $F_\e$. Then there exists $\alpha\in\lbrace \pm 1\rbrace$ and $x_*\in\R^2$ such that
\begin{align*}
\int_\Omega \left| m(x)-\alpha \, i \frac{x-x_*}{|x-x_*|}\right|^4\, dx
\leq c\, 
\, \liminf_{\e\to 0}  F_\e(m_\e;\Omega)^{\frac 23},
\end{align*}
for some constant $c=c(K)>0$.
\end{cor}

\begin{rem}\label{r:lorent}
In comparison with the estimate \eqref{eq:estimlorent} for $E_\e^{AG}$ from \cite{lor14}, we don't require $\Omega$ to be convex, and impose only the boundary condition $m\cdot n_{\partial\Omega}=0$ on limit maps. 
However, we only obtain bounds in the limit $\e\to 0$, while \eqref{eq:estimlorent} is valid for any fixed $\e>0$. 
Note that the constant $c$ in \eqref{eq:main} depends on $K$, while the constant $C$ in \eqref{eq:estimlorent} is absolute; 
on the other hand it is not possible to obtain an absolute constant if we drop the assumption of $\Omega$ being convex.
Indeed if $\Omega_{\delta}=B_1((-1+\delta,0)) \cup B_1((1-\delta,0))$ 
(or rather, a mollification of this domain at scale much smaller than $\delta$),
then  $\liminf_{\e\to 0} \inf F_\e(\cdot, \Omega_{\delta})$ tends to 0 as $\delta\to 0$.
This can be checked by using the solution of \eqref{eq:eik} in $\Omega_\delta$ given by $m_\delta=i\nabla d_\delta$, where $d_\delta(x)=\dist(x,\partial\Omega_\delta)$, and the upper bound (see e.g. \cite{contidel,poliakovsky07}) $\liminf_{\e\to 0} \inf F_\e(\cdot,\Omega_\delta)\leq C \int_{J_\delta} |[m_\delta]|^3 d\mathcal H^1$, where $J_\delta$ is the jump set of $m_\delta$.
\end{rem}

Our proofs of Theorem~\ref{t:main} and its corollaries rely on a generalization of the zero-energy kinetic equation \eqref{eq:kinzero}   to  limits $m=\lim_{\e\to 0} m_\e$ of bounded energy sequences:
\begin{align}\label{eq:kin}
&e^{is}\cdot\nabla_x \mathbf 1_{m(x)\cdot e^{is}>0} =\partial_s\sigma,\qquad \sigma\in\mathcal M(\Omega\times \R/2\pi\Z),\\
& |\sigma|(\Omega\times\R/2\pi\Z) \leq c_0 \liminf_{\e\to 0} F_\e(m_\e;\Omega),\nonumber
\end{align}
where $c_0>0$ is an absolute constant.
This kinetic formulation, inspired by the field of scalar conservation laws \cite{lionsperthametadmor94}, was first obtained in  \cite{jabinperthame01} for the Aviles-Giga functional \eqref{eq:AG} (see also \cite{GL}) and in \cite{RS03} for the micromagnetics model \eqref{eq:RS}. It also applies to the more general functional $F_\e$ (see Appendix~\ref{a:ent}). It is worth noting that it implies that $m$ admits strong traces along $1$-rectifiable subsets (see \cite{vasseur01} or \cite{ODL03}), and in particular along $\partial\Omega$.

Among the measures $\sigma$ satisfying  \eqref{eq:kin}, 
we consider the measure $\sigma_{\mathrm{min}}$ with minimal total variation
$|\sigma|(\Omega \times \R/2\pi \Z)$ (the uniqueness of $\sigma_{\min}$ is proven in \cite{marconi21ellipse}), and set
\begin{equation}\label{e:def_nu}
\nu= (p_x)_\sharp |\sigma_{\min}|,
\end{equation}
where $p_x:\Omega\times \R/2\pi\Z \to \Omega$ denotes the standard projection. In particular we have 
\begin{align}\label{eq:nulowF}
\nu(\Omega)=|\sigma_{\min}|(\Omega\times\R/2\pi\Z)
\leq c_0 \liminf_{\e\to 0} F_\e(m_\e;\Omega).
\end{align}
With these notations we may reformulate our main estimate  as follows.

\begin{thm}\label{t:main2}
Let $\Omega\subset\R^2$ be a simply connected  open set of class $C^{1,1}$ with $\mathcal H^1(\partial\Omega)=2\pi$ and 
$\sup_{\partial \Omega} |\kappa| \le K$ for some $K>0$, where $\kappa$ denotes the curvature of $\partial \Omega$.
If there exists $m\colon\Omega\to\R^2$ solving the eikonal equation \eqref{eq:eik} and the kinetic formulation \eqref{eq:kin}, then
\begin{align}\label{eq:main2}
\inf_{x_*\in\R^2}\int_{\partial\Omega}\left| n_{\partial\Omega}(x)-\frac{x-x_*}{|x-x_*|}\right|^2 d \mathcal H^1(x)\leq c \, \nu(\Omega),
\end{align}
for some constant $c>0$ depending only on $K$.
\end{thm}

Theorem~\ref{t:main2} implies Theorem~\ref{t:main} thanks to \eqref{eq:nulowF}.
Similarly, Corollaries~\ref{c:estOmega} and \ref{c:estm} will be consequences of the estimates
\begin{align}
\dist(\partial\Omega,\partial B_1(x_*)) &\leq c \, \nu(\Omega)^{\frac 12},
\label{eq:estOmega}
\\
\int_\Omega \left| m(x)-\alpha \, i \frac{x-x_*}{|x-x_*|}\right|^4\, dx
&\leq c\, 
\, \nu(\Omega)^{\frac 23},
\label{eq:estm}
\end{align}
for some $x_*\in\R^2$ and $\alpha\in\lbrace \pm 1\rbrace$. Next we briefly describe our strategy to prove Theorem~\ref{t:main2}.

\subsection{Strategy of proof}

\subsubsection{A basic geometric argument}

At the heart of our estimates is the following basic geometric argument. Assume $m$ is a zero-energy state, that is, a solution of \eqref{eq:eik} and \eqref{eq:kinzero}, and assume moreover that $m=-\tau$ on $\partial\Omega$. Suppose there are three boundary points $x_k\in\partial\Omega$, $k=1,2,3$, and three directions $e^{i\alpha_k}$ with the following properties:
\begin{enumerate}
\item the three lines $x_k +e^{i\alpha_k}\R$ intersect at a point $z_0\in\Omega$,
\item the direction $e^{i\alpha_k}$ points in the half-cirle determined by the direction $m=-\tau$ at $x_k$, i.e. $e^{i\alpha_k}\cdot \tau(x_k) < 0$,
\item the three directions $e^{i\alpha_k}$ are not contained in the same half-circle.
\end{enumerate}
Such configuration is made impossible by the kinetic equation \eqref{eq:kinzero}, because $\mathbf 1_{m\cdot e^{i\alpha_k}>0}$ must be constant along the line $x_k +\R e^{i\alpha_k}$. By the second property, its constant value must be one for $k=1,2,3$, which implies that $m(z_0)$ has positive scalar product with the three directions $e^{i\alpha_k}$, which is impossible by the third property.
(To make this rigorous actually requires a bit of care and `almost everywhere' statements, as in \cite{JOP}.)
So there are no triplets of points satisfying that condition, and this can be seen to imply that $\partial\Omega$ must be a circle, as it forces the normal lines at any three boundary points to be concurrent.

\subsubsection{A quantitative version}

Our strategy is to make that basic geometric argument quantitative.
Let $a(x_1,x_2,x_3)\geq 0$ quantify the above properties:  $a>0$ if there are three lines from $x_k$ with directions $e^{i\alpha_k}$   intersecting well inside $\Omega$, with $e^{i\alpha_k}\cdot \tau(x_k) \leq -a$ and the three directions are not contained in the $a$-neighborhood of any half-circle.
Note that this is a purely geometric quantity,  defined without any reference to a map $m$.
Let $m$ satisfy the eikonal equation \eqref{eq:eik} and kinetic equation \eqref{eq:kin} with a non-zero dissipation measure $\nu(\Omega)=|\sigma_{\min}|(\Omega\times\R/2\pi\Z)$. 
Compared to the above basic geometric argument, the assumptions on $m$ are relaxed in two ways: $\nu(\Omega)>0$, and the trace $m_{\lfloor \partial\Omega}$ can take values into $\lbrace \pm \tau\rbrace$.
Then we show that
\begin{align}\label{eq:estanu}
\int_{\partial\Omega^3}a(x_1,x_2,x_3)^2\, d (\mathcal H^1)^{\otimes 3} \leq c\, \nu(\Omega) +c\, \mathcal H^1(\lbrace m_{\lfloor\partial\Omega}=\tau\rbrace),
\end{align}
provided $\Omega$ is a priori close enough to a disk.
This a priori  condition will be satisfied if $\nu(\Omega)$ is small enough thanks to a compactness argument and the characterization of zero-energy states \cite{JOP}.
To deal with the trace issue, \eqref{eq:estanu} needs to be complemented with the estimate
\begin{align}\label{eq:traceestim}
\mathcal H^1(\lbrace m_{\lfloor\partial\Omega}=\tau\rbrace) \leq c \, \nu(\Omega),
\end{align} 
provided the left-hand side is a priori small enough.
Again, this a priori condition can be obtained by means of a compactness argument and the characterization of zero-energy states. 
(The compactness argument tells us that, for small $\nu(\Omega)$, one of the complementary subsets $\lbrace m_{\lfloor \partial\Omega}=\tau\rbrace$ or $\lbrace m_{\lfloor \partial\Omega}=-\tau\rbrace$  is small, here we consider without loss of generality only the first case.)
Finally Theorem~\ref{t:main} is obtained by
estimating the deviation of $n_{\partial\Omega}$ from the disk's normal
 with the geometric quantity $a$, which relies on purely geometric considerations (that is, independent of the map $m$).

\subsubsection{Lagrangian representation}

The quantitative estimate \eqref{eq:estanu} is our main new ingredient.
It relies on the Lagrangian representation introduced by the second author in \cite{marconi21ellipse,marconi22structure}, 
which allows to  decompose the dissipation $\nu(\Omega)$ along Lagrangian trajectories. 
Roughly speaking, the dissipation created by one trajectory is the amount by which it deviates from being a straight line. In particular, absence of dissipation ($\nu=0$) is equivalent to Lagrangian trajectories being straight lines. 
With this interpretation in mind, the intuition behind the proof of \eqref{eq:estanu} can be  explained as follows. 
Assume for simplicity that $m_{\lfloor\partial\Omega}=-\tau$.
The basic geometric argument outlined above implies that Lagrangian trajectories meeting three boundary points $x_k$  with directions  close to $e^{i\alpha_k}$   cannot be straight lines if $a>0$:
 they must therefore create dissipation. 
More precisely, for intervals of directions of order $a$ around each $e^{i\alpha_k}$, at least one of the corresponding three trajectories should deviate of order $a$ from being a straight line, and summing these contributions provides a dissipation of order $a^2$, as expressed by \eqref{eq:estanu}.
Many technical details are however needed to make this intuition rigorous. 
In particular, trajectories cannot be considered individually, but in `packets' inside which only a certain amount of trajectories follow that intuition.
Similar arguments are used to prove the trace estimate \eqref{eq:traceestim}.
 
 \subsection{Plan of the article}

The article is organized as follows. In Section~\ref{s:geom} we gather purely geometrical estimates, showing in particular that \eqref{eq:estanu}-\eqref{eq:traceestim} imply Theorem~\ref{t:main} and Corollary~\ref{c:estOmega}. 
In Section~\ref{s:dissip} we prove \eqref{eq:estanu}, 
under the a priori assumption that $\Omega$ is close to a disk.
In Section~\ref{s:comp} we present the compactness argument that allows to lift that a priori assumption.
In Section~\ref{s:trace} we prove the trace estimate \eqref{eq:traceestim}.
In the short Section~\ref{s:proofmain} we gather all previous results to prove Theorem~\ref{t:main2}, Theorem~\ref{t:main} and Corollary~\ref{c:estOmega}.
In Section~\ref{s:proofestm} we prove Corollary~\ref{c:estm}.
In Section~\ref{s:sharp} we prove the sharpness statement of Proposition~\ref{p:sharp}.
In Appendix~\ref{a:ent} we recall the arguments leading to the kinetic formulation~\ref{eq:kin}, showing in particular that they apply to our generalized functional $F_\e$.
In Appendix~\ref{a:ARS} we recall some of the analysis of the model \eqref{eq:ARS} from \cite{ARS},  to emphasize that in that case the total dissipation $\nu(\Omega)$ provides a sharp lower bound.
In Appendix~\ref{a:alt} we present a quantitative proof which allows to bypass the compactness argument of Section~\ref{s:comp} under the additional assumption that $m=-\tau$ on $\partial\Omega$, an assumption relevant for the Aviles-Giga model \eqref{eq:AG} but not for the other models considered here.

\subsection{Notations}

We use the symbol $\lesssim$ to denote inequality up to an absolute multiplicative constant and we write $a \sim b$ if both $a\lesssim b$ and $b\lesssim a$ hold true.
We  systematically identify $\R^2$ and $\mathbb C$, multiplication by $i$ corresponds to rotation by an angle $\pi/2$.
We denote by $\p\colon\R/2\pi\Z\to\partial\Omega$ a $C^{1,1}$ counterclockwise arc-length parametrization of $\partial\Omega$, and by $\tau(\p(s))=\dot\p(s)$, $n_{\partial\Omega} =-i\tau $ the corresponding unit tangent and normal.

\section{Geometric estimates}\label{s:geom}

Here and in the rest of the article, we fix $B_R(x_0)$ a maximal disk contained in $\Omega$.
As explained in the introduction, the proofs of our main results rely on a geometric quantity $a$ defined for triples of boundary points.

\begin{definition}\label{def:a}
Given $\hat x = (x_1,x_2,x_3) \in \partial\Omega^3$, we define $a(\hat x)\geq 0$ as  the maximal value $a\geq 0$
for which there are 
 $\alpha_1,\alpha_2,\alpha_3 \in \R/2\pi\Z$ such that 
\begin{enumerate}
\item the lines $x_k + e^{i\alpha_k}  \R$ are concurrent in $B_{R/2}(x_0)$, namely there are $t_1,t_2,t_3 \in \R$ such that 
\[
x_1 + t_1 e^{i\alpha_1} = x_2 + t_2 e^{i\alpha_2} =x_3 + t_3 e^{i\alpha_3}=z_0 \in B_{R/2}(x_0);
\]
\item $\min(0,\tau (x_k) \cdot e^{i\alpha_k}) \leq -a$ for $k=1,2,3$;
\item $a \le \max\{ l(\alpha_1,\alpha_2,\alpha_3) - \pi, 0 \}$, where $l(\alpha_1,\alpha_2,\alpha_3)$ denotes the length of the shortest interval in $\R/2\pi\Z$ containing $\alpha_1,\alpha_2,\alpha_3$.
\end{enumerate}
Note that each direction $e^{i\alpha_k}$ may be entering, i.e. $t_k>0$, or exiting, i.e. $t_k<0$ (equivalently, $(x_k-z_0)\cdot\tau(x_k)>0$ or $(x_k-z_0)\cdot\tau(x_k)<0$). 
\end{definition}

We observe that $a(\cdot)$ is identically 0 if $\Omega$ is a disk. 
A useful geometric interpretation of $a( \cdot )$ is that $a(x_1,x_2,x_3)$ is bounded below
by the inner radius of the
triangle formed by the three normals to $\partial \Omega$ passing through $x_1$, $x_2$ and $x_3$. See Figure~\ref{f:a}

\begin{figure}[h]
\centering
\def\svgwidth{0.55\columnwidth}
\begingroup%
  \makeatletter%
  \providecommand\color[2][]{%
    \errmessage{(Inkscape) Color is used for the text in Inkscape, but the package 'color.sty' is not loaded}%
    \renewcommand\color[2][]{}%
  }%
  \providecommand\transparent[1]{%
    \errmessage{(Inkscape) Transparency is used (non-zero) for the text in Inkscape, but the package 'transparent.sty' is not loaded}%
    \renewcommand\transparent[1]{}%
  }%
  \providecommand\rotatebox[2]{#2}%
  \newcommand*\fsize{\dimexpr\f@size pt\relax}%
  \newcommand*\lineheight[1]{\fontsize{\fsize}{#1\fsize}\selectfont}%
  \ifx\svgwidth\undefined%
    \setlength{\unitlength}{162.60761608bp}%
    \ifx\svgscale\undefined%
      \relax%
    \else%
      \setlength{\unitlength}{\unitlength * \real{\svgscale}}%
    \fi%
  \else%
    \setlength{\unitlength}{\svgwidth}%
  \fi%
  \global\let\svgwidth\undefined%
  \global\let\svgscale\undefined%
  \makeatother%
  \begin{picture}(1,0.78321193)%
    \lineheight{1}%
    \setlength\tabcolsep{0pt}%
    \put(0,0){\includegraphics[width=\unitlength,page=1]{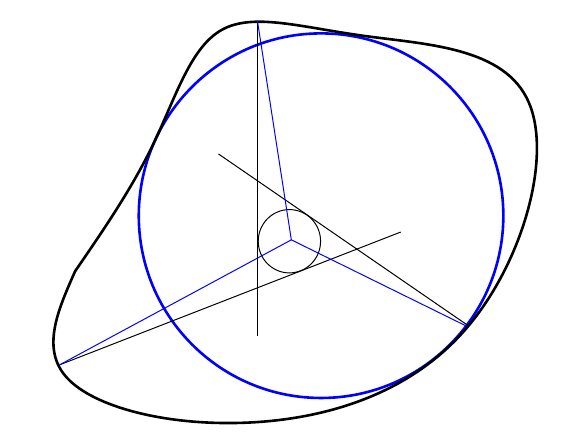}}%
    \put(0.48370742,0.36300949){\makebox(0,0)[lt]{\lineheight{1.25}\smash{\begin{tabular}[t]{l}$z_0$\end{tabular}}}}%
    \put(0.04092376,0.11855596){\makebox(0,0)[lt]{\lineheight{1.25}\smash{\begin{tabular}[t]{l}$x_1$\end{tabular}}}}%
    \put(0.43297181,0.76428218){\makebox(0,0)[lt]{\lineheight{1.25}\smash{\begin{tabular}[t]{l}$x_2$\end{tabular}}}}%
    \put(0.84808152,0.16467924){\makebox(0,0)[lt]{\lineheight{1.25}\smash{\begin{tabular}[t]{l}$x_3$\end{tabular}}}}%
    \put(0.87114313,0.69509722){\makebox(0,0)[lt]{\lineheight{1.25}\smash{\begin{tabular}[t]{l}$\Omega$\end{tabular}}}}%
    \put(0,0){\includegraphics[width=\unitlength,page=2]{a.pdf}}%
    \put(-0.00186175,0.02755095){\makebox(0,0)[lt]{\lineheight{1.25}\smash{\begin{tabular}[t]{l}$\tau(x_1)$\end{tabular}}}}%
  \end{picture}%
\endgroup%

\caption{\small The black lines through the three points $x_1,x_2,x_3 \in \partial \Omega$ are the normals to $\partial \Omega$, while the blue lines have directions $e^{i\alpha_1}, e^{i\alpha_2},e^{i\alpha_3}$ and they are concurrent in the point $z_0$ as in the definition of $a$. In this case $z_0$ is chosen as the center of the incircle of the triangle formed by the normals, and $a$ is of the order of that incircle's radius.}
\label{f:a}
\end{figure}

The quantity $a$ defined in Definition~\ref{def:a} will be useful only if the three segments $[z_0,x_k]$ are contained in $\Omega\cup \lbrace x_k\rbrace$. That is why we define next subsets of $\partial\Omega$ where this will be ensured. Recall that $B_R(x_0)$ is a maximal disk contained in $\Omega$, and consider the set
\[
E_*= \{ x \in \partial \Omega\colon t x + (1-t) x_0 \in \Omega \; \forall t \in (0,1) \},
\]
in some sense the part of the boundary that is star-shaped around $x_0$. And
for every $\eta>0$, we define the subset of $E_*$ given by
\begin{align}\label{eq:Eeta}
E(\eta)= \{ x \in E_* \colon |x-x_0| \leq (1+\eta )R\}.
\end{align}
The main aim of the present section is to prove that the quantity $a$ can be used to estimate the deviation of $\Omega$ from a disk, as follows.

\begin{prop}\label{p:est_a_circle}
Let $\Omega$ as in Theorem~\ref{t:main}.
There exists $\eta_0=\eta_0(K)>0$ such that, if $E(\eta_0)=\partial\Omega$ then 
\[
\dist^2(\partial\Omega,\partial D_1(x_*)) +
\int_{\partial\Omega}\left| n_{\partial\Omega}(x)-\frac{x-x_*}{|x-x_*|}\right|^2 
\lesssim \int_{\partial\Omega^3} a^2 d (\mathcal H^1)^{\otimes 3}.
\]
for some $x_*\in\R^2$.
\end{prop}

\subsection{A few preliminary geometric facts}

First we show that boundary points close to the maximal inscribed circle must have a unit normal close to radial (with respect to the inscribed circle's center).

\begin{lem}\label{l:geom1}
Let $\Omega$ be a $C^{1,1}$ simply connected domain with $\mathcal H^1(\partial\Omega)=2\pi$, $\sup_{\partial\Omega}|\kappa|\leq K$ and denote by $B_R(x_0)$ a maximal disk contained in $\Omega$. Then $1/K\leq R \leq 1$ and  for every $x \in \partial \Omega$ we have
\[
\left| \tau(x) \cdot \frac{x-x_0}{|x-x_0|}\right| \le 2 \sqrt{ K \dist (x,\partial B_R(x_0)) }.
\]
\end{lem}
\begin{proof}[Proof of Lemma~\ref{l:geom1}]
The isoperimetric inequality ensures $R\leq 1$.
For a proof of the property $R\ge 1/K$ we refer to \cite{pestov-ionin,howard-treibergs}.
Let us consider an arc-length parametrization $\p: \R/2\pi\Z \to \R^2$ of $\partial \Omega$ and let $\psi\colon \R/2\pi\Z \to \R$ be defined by 
\begin{equation*}
\psi(s) = |\p(s) - x_0| - R = \dist ( \p(s), \partial B_R(x_0)).
\end{equation*}
In particular we have
\begin{equation}\label{e:der_psi}
\begin{split}
\psi'(s)  & = \frac{\dot \p(s) \cdot (\p(s) - x_0)}{|\p(s) - x_0|} = \tau (\p(s)) \cdot  \frac{\p(s) - x_0}{|\p(s) - x_0|},\\
 \psi''(s) &=\frac{\ddot\p(s)\cdot (\p(s) - x_0)}{|\p(s) - x_0|} - 
\frac{|\dot \p (s) \cdot (\p(s) - x_0)|^2}{|\p(s)-x_0|^3}.
\end{split}
\end{equation}
and therefore 
\begin{equation*}
\|\psi''\|_{L^\infty} \le K + \frac{1}{R} \le 2K.
\end{equation*}
Now consider, for any $a>2$, the function
\begin{align*}
\varphi(s)=a\,\|\psi''\|_{L^\infty} \psi(s) -\psi'(s)^2,
\end{align*}
which is $C^1$ with derivative
\begin{align*}
\varphi'(s)&=a\,\|\psi''\|_{L^\infty} \psi'(s)-2\psi'(s)\psi''(s)\\
&=\left(a\,\|\psi''\|_{L^\infty}-2\psi''(s)\right) \psi'(s).
\end{align*}
The first factor is positive since $a>2$. Hence at a minimal point $s_0$ of $\varphi$ one must have $\psi'(s_0)=0$ and so $\varphi(s_0)=a\,\|\psi''\|_{L^\infty} \psi(s_0)\geq 0$. Therefore $\varphi$ is a nonnegative function. As this is valid for any $a>2$ we deduce that
\begin{align*}
\psi'(s)^2\leq 2 \|\psi''\|_{L^\infty} \psi(s)\leq 4K \, \psi(s).
\end{align*}
Taking the square root and recalling the expression of $\psi$ and $\psi'$ concludes the proof.
\end{proof}

The next lemma ensures that $a(\hat x)$ is meaningful whenever $\hat x\in E(\eta)^3$ for sufficiently small $\eta>0$.

\begin{lem}\label{l:segment}
Let $\Omega$ be a $C^{1,1}$ simply connected domain with $\mathcal H^1(\partial\Omega)=2\pi$ and $\sup_{\partial\Omega}|\kappa|\leq K$.
There exists $\eta_0>0$ depending only  on $K$ such that, for all $x\in E(\eta_0)$ and any $z\in B_{2R/3}(x_0)$, the segment $[z,x]$ is included in $\Omega\cup\lbrace x\rbrace$. 
\end{lem}
\begin{proof}[Proof of Lemma~\ref{l:segment}]
Let $x\in E(\eta_0)$, and write $x_0=x+re^{i\theta_0}$ for some $r\in [R,(1+\eta_0)R]$ and $\theta_0\in\R$. Denote by $L_\theta=x+[0,\infty)e^{i\theta}$ the half line from $x$ in direction $\theta$. This half-line has a nontrivial intersection with $B_{2R/3}(x_0)$ if and only if $\theta\in (\theta_0-\alpha,\theta_0+\alpha) +2\pi\Z$, 
where $\alpha=\arcsin(2R/(3r))=\arcsin(2/3) +\mathcal O(\eta_0)\leq\pi/4$ if $\eta_0$ is small enough.
By definition \eqref{eq:Eeta} of $E(\eta)$ and of $E_*$, for all $z\in B_{2R/3}(x_0)\cap L_{\theta_0}$ the segment $[z,x]$ is included in $\Omega\cup\lbrace x\rbrace$. 
If the conclusion of Lemma~\ref{l:segment} is not true, by continuity we may therefore find $\theta_1\in (\theta_0-\alpha,\theta_0+\alpha)$, $\theta_1\neq\theta_0$, and $z_1\in B_{2R/3}(x_0)\cap L_{\theta_1}$ such that $[z_1,x]$ is not included in $\Omega\cup\lbrace x\rbrace$, while that property holds for $z\in B_{2R/3}(x_0)\cap L_\theta$ if $\theta\in (\theta_0,\theta_1)$. 
This implies the existence of $y \in [z_1,x]\cap\partial\Omega\setminus\lbrace x\rbrace$, with tangent vector $\tau(y)=e^{i\theta_1}$ and  $|x-y|\lesssim R\eta_0$.
In particular we have $(y-x_0)/|y-x_0|=e^{i\theta}$ with $|\theta-\theta_0|\lesssim R\eta_0$, and by Lemma~\ref{l:geom1} applied to the boundary point $y$ we infer
\begin{align*}
|e^{i\theta_1}\cdot e^{i\theta_0}|^2\lesssim  KR\eta_0\lesssim c_0K\eta_0.
\end{align*}
As $|e^{i\theta_1}\cdot e^{i\theta_0}|^2 \geq \cos^2\alpha \geq 1/2$ this implies $\eta_0\gtrsim 1/K$, hence choosing $\eta_0=1/(CK)$ for a large enough absolute constant $C$ ensures the validity of Lemma~\ref{l:segment}.
\end{proof}

We also remark that, for a connected component of $\partial\Omega \cap \overline B_{(1+\eta)R}(x_0)$ to be contained in $E(\eta)$, it is sufficient that one of its elements belongs to $E(\eta)$.

\begin{lem}\label{l:continuation}
Let $\eta \in [0, 1/(4K)]$ and assume that $\bar x = \p(\bar s) \in E(\eta)$ and $s_1,s_2 \in \R $ are such that 
\[
\bar s\in [s_1,s_2] \qquad \mbox{and} \qquad |\p(s) - x_0| \leq  (1+\eta)R \quad \forall s \in [s_1,s_2].
\]
Then $\p((s_1,s_2))\subset E(\eta)$.
\end{lem}
\begin{proof}[Proof of Lemma~\ref{l:continuation}]
Since we assume that $|\p(s) - x_0| \le (1+\eta)R$ for all $s\in [s_1,s_2]$, it only remains to show that $\p((s_1,s_2))\subset E_*$, that is, for all $s\in (s_1,s_2)$ the line interval $\lbrace (1-t)x_0 +t \p(s)\colon 0<t<1\rbrace$ is contained in $\Omega$. Consider the largest interval $I\subset (s_1,s_2)$ containing $\bar s$ and such that $\p(s)\in E_*$ for all $s\in I$. Note that $I$ is open and non-empty. Assume by contradiction that $(s_1,s_2)\setminus I\neq \emptyset$, and denote by $\tilde s\in (s_1,s_2)$ an extremity of $I$. Then by maximality of $I$ the line interval $\lbrace (1-t)x_0 +t \p(\tilde s)\colon 0<t<1\rbrace$ intersects $\partial\Omega$: there exists $t\in (0,1)$ such that $\tilde x=(1-t)x_0 +t \p(\tilde s) \in\partial\Omega$. By definition of $I$, locally near $\tilde x$ the $C^2$ curve $\partial\Omega$ stays on one side of the line $x_0 +\R (\p(\tilde s)-x_0)$, hence it must be tangent to that line. Therefore we have $\tau(x)=\pm (\tilde x-x_0)/|\tilde x -x_0|$. By Lemma~\ref{l:geom1}, and since $R\leq |\tilde x-  x_0|<|\p(\tilde s)-x_0|\leq (1+\eta)R$  this implies $2\sqrt {K\eta R}>1$, in contradiction with the assumption that $\eta \in [0, 1/(4K)]$ and the fact that $R\leq 1$ (by isoperimetric inequality).
\end{proof}

Finally we remark that the function $a$ is Lipschitz.

\begin{lem}\label{l:a_Lip}
The function $a$ is Lipschitz on $\partial\Omega^3$ (with respect to the geodesic distance), with Lipschitz constant $L\lesssim K$.
\end{lem}
\begin{proof}
Let $\hat x=(x_1,x_2,x_3)\in \partial\Omega^3$, and $\hat\alpha =(\alpha_1,\alpha_2, \alpha_3)\in (\R/2\pi\Z)^3$ as in the definition of $a(\hat x)$. Denote by $z_0\in B_{R/2}(x_0)$ the intersection point of the three lines $x_k+e^{i\alpha_k}\R$. Let $\hat x'=(x'_1,x'_2,x'_3)\in \partial\Omega^3$. Since $z_0$ lies at a distance at least $R/2$ of each $x_k$, the three concurrent lines connecting $x'_k$ to $z_0$ are of the form $x'_k+e^{i\alpha'_k}\R$ for some $\hat\alpha' =(\alpha'_1,\alpha'_2, \alpha'_3)\in (\R/2\pi\Z)^3$ such that
\begin{align*}
|\alpha_k-\alpha'_k |\lesssim \frac 1R |x_k-x_k'| \lesssim K \dist(\hat x,\hat x')
\qquad\forall k\in\lbrace 1,2,3\rbrace.
\end{align*}
Therefore we have
\begin{align*}
\max(l(\hat\alpha')-\pi,0)
&\geq 
\max(l(\hat\alpha)-\pi,0) - C K \dist(\hat x,\hat x')\\
&\geq a(\hat x) - C K \dist(\hat x,\hat x'),
\end{align*}
for some absolute constant $C>0$. Moreover by definition of $K$ we have
\begin{align*}
|\tau(x_k)-\tau(x'_k)|\lesssim K\dist(\hat x,\hat x')\qquad\forall k\in\lbrace 1,2,3\rbrace,
\end{align*}
and therefore
\begin{align*}
\tau(x'_k)\cdot e^{i\alpha_k'}
&\leq 
\tau(x_k)\cdot e^{i\alpha_k} +C  |\tau(x_k) -\tau(x_k')| + C |\alpha_k-\alpha_k'|
\\
& \leq -a(\hat x) - C K \dist(\hat x,\hat x').
\end{align*}
This shows that
\begin{align*}
a(\hat x')\geq a(\hat x) - C K \dist(\hat x,\hat x').
\end{align*}
Exchanging the roles of $\hat x$ and $\hat x'$ we conclude that $|a(\hat x')- a(\hat x) |\lesssim  K \dist(\hat x,\hat x')$.
\end{proof}

\subsection{Proof of Proposition~\ref{p:est_a_circle}}

We start by remarking that the distance between $\partial\Omega$ and a unit circle is controlled by the $L^1$-difference of their normals.

\begin{lem}\label{l:aux}
If $\Omega$ is a simply connected $C^1$ domain such that $\mathcal H^1(\partial\Omega)=2\pi$, then for any $x_*\in\Omega$ such that $\Omega$ is strictly star-shaped around $x_*$, we have
\begin{align*}
\dist(\partial\Omega,\partial D_1(x_*))
& \leq  \int_{\partial\Omega}\left|n_{\partial\Omega}(x)-\frac{x-x_*}{|x-x_*|}\right| \, d\mathcal H^1(x)
\end{align*}
\end{lem}
\begin{proof}[Proof of Lemma~\ref{l:aux}]
We choose coordinates in which $x_*=0$ and denote
\begin{align*}
n_*(x)=\frac{x-x_*}{|x-x_*|}=\frac{x}{|x|}.
\end{align*}
First we claim that
\begin{align}\label{eq:taun}
\left|\tau_{\partial\Omega}(x)\cdot n_*(x)\right|
\leq
\left|n_{\partial\Omega}(x)-n_*(x)\right|.
\end{align}
To prove \eqref{eq:taun}, note that since $\Omega$ is strictly star-shaped around $x_*$,  i.e. $n_{\partial\Omega}\cdot n_*>0$ on $\partial\Omega$, we have
\begin{align*}
n_{\partial\Omega}\cdot n_* =\sqrt{1-(\tau_{\partial\Omega}\cdot n_*)^2}
\end{align*}
Hence we deduce
\begin{align*}
|n_{\partial\Omega}-n_*|^2=2-2\, n_{\partial\Omega}\cdot n_*
=2-2\sqrt{1-(\tau_{\partial\Omega}\cdot n_*)^2}.
\end{align*}
Estimate \eqref{eq:taun} follows from this identity and the convexity inequality
\begin{align*}
2-2\sqrt{1-t} \geq t\qquad\forall t\in [0,1].
\end{align*}
Let $\p\in C^1(\R/2\pi\Z;\R^2)$ denote a counterclockwise arc-length parametrization of $\partial\Omega$, and let $r_{\min}=\min |\p|$, $r_{\max}=\max |\p|$ be the respective radii of the maximal centered disk contained in $\Omega$ and of the minimal centered disk containing $\Omega$. As
\begin{align*}
\frac{d}{ds}|\p(s)|=\dot\p(s)\cdot \frac{\p(s)}{|\p(s)|}
=\tau_{\partial\Omega}(\p(s))\cdot n_*(\p(s)),
\end{align*}
we infer, using also \eqref{eq:taun},
\begin{align}\label{eq:oscrad}
r_{\max}-r_{\min} \leq \frac12 \int_{\partial\Omega} |\tau_{\partial\Omega}\cdot n_*|\, d\mathcal H^1\leq \frac12\int_{\partial\Omega} |n_{\partial\Omega}- n_*|\, d\mathcal H^1.
\end{align}
Note that $r_{\min}\leq 1$ thanks to the isoperimetric inequality, so if $r_{\max}\geq 1$ then
\eqref{eq:oscrad} directly implies the conclusion of Lemma~\ref{l:aux}. In what follows we may therefore assume $r_{\max} < 1$. As $\Omega$ is strictly star-shaped around $x_*=0$, the map
\begin{align*}
\p_*(s)=\frac{\p(s)}{|\p(s)|},\qquad s\in\R/2\pi\Z,
\end{align*}
defines a one-to-one parametrization of the unit circle $\partial D_1$. In particular we must have
\begin{align*}
\int_{\R/2\pi\Z} |\dot\p_*| =2\pi.
\end{align*}
On the other hand direct calculation shows
\begin{align*}
|\p| \, \dot\p_* =\dot\p - (\dot\p\cdot\p_*)\,\p_*,
\end{align*}
hence
\begin{align*}
\int_{\R/2\pi\Z}(1-|\p|)|\dot\p_*| 
&
=2\pi - \int_{\R/2\pi\Z} | \dot\p - (\dot\p\cdot\p_*)\,\p_* | \\
&\leq
2\pi - \int_{\R/2\pi\Z}|\dot\p| + \int_{\R/2\pi\Z} |\dot\p\cdot\p_*| 
\\
&
 =\int_{\partial\Omega} |\tau_{\partial\Omega}\cdot n_*|\, d\mathcal H^1.
\end{align*}
As $0 < 1-r_{\max} \leq 1-|g|$, this implies
\begin{align*}
2\pi(1-r_{\max})=(1-r_{\max})\int_{\R/2\pi\Z}|\dot\p_*| \leq
\int_{\partial\Omega} |\tau_{\partial\Omega}\cdot n_*|\, d\mathcal H^1.
\end{align*}
Together with \eqref{eq:oscrad} this gives
\begin{align*}
1-r_{\min} \leq \left(\frac12+\frac{1}{2\pi}\right)
\int_{\partial\Omega} |\tau_{\partial\Omega}\cdot n_*|\, d\mathcal H^1
\leq
\int_{\partial\Omega} |n_{\partial\Omega}- n_*|\, d\mathcal H^1,
\end{align*}
proving Lemma~\ref{l:aux} also in the case  $r_{\max}<1$.
\end{proof}

Now we turn to the proof of Proposition~\ref{p:est_a_circle}.

\begin{proof}[Proof of Proposition~\ref{p:est_a_circle}]
We choose coordinates in which $x_0=0$.
We assume $\Omega=E(\eta_0)$ for some $\eta_0=\e_0^2/K>0$, with $\e_0$ to be fixed later. Hence $\Omega$ is star-shaped around $x_0$ and $B_R(x_0)\subset\Omega\subset B_{(1+\eta_0)R}$.

Recall $R\leq 1$ by the isoperimetric inequality, and thanks to Lemmas~\ref{l:geom1} and \ref{l:aux} we have 
\begin{align}\label{eq:Rcloseto1}
\sup_{x\in\partial\Omega} \left|n_{\partial\Omega}(x)-\frac{x }{|x|}\right|  \lesssim \e_0,
\qquad
0\leq 1-R \lesssim \e_0.
\end{align}
For any $\alpha\in\R/2\pi\Z$ we denote by $I_\alpha$ the portion of $\partial\Omega$ that intersects the centered cone corresponding to angles from $\alpha$ to $\alpha +\pi/6$, that is
\begin{align*}
I_\alpha =\left\lbrace x\in\partial\Omega\colon \frac{x}{|x|}=e^{i\theta}\text{ for some }\theta\in [\alpha,\alpha +\pi/6]\right\rbrace. 
\end{align*}
Thanks to the above we have $\mathcal H^1(I_\alpha)=\pi/6 +\mathcal O(\e_0)\geq \pi/12$ for small enough $\eta_0$, so by Fubini there exist $\bar x_2\in I_{\alpha +2\pi/3}$, $\bar x_3\in I_{\alpha +4\pi/3}$ such that
\begin{align*}
\int_{I_\alpha} a^2(x_1,\bar x_2,\bar x_3)\,d\mathcal H^1(x_1)
\lesssim
\int_{\partial\Omega^3}a^2 \, d(\mathcal H^1)^{\otimes 3}.
\end{align*}
Denote by $z_\alpha$ the intersection of the two normal lines at $\bar x_2$ and $\bar x_3$. It satisfies $|z_\alpha| \lesssim \e_0$. Let $n_\alpha$ denote the vortex centered at $z_\alpha$, that is,
\begin{align*}
n_\alpha(x)=\frac{x-z_\alpha}{|x-z_\alpha|}.
\end{align*}
We claim that
\begin{align}\label{eq:aboundsnormal}
\left|n_{\partial\Omega}(x_1)-n_\alpha(x_1)\right|
\lesssim a(x_1,\bar x_2,\bar x_3)\qquad\forall x_1\in I_\alpha.
\end{align}
Let indeed $x_1\in I_\alpha$. 
We denote by $L_1,L_2,L_3$ the normal lines to $\partial\Omega$ at $x_1,\bar x_2,\bar x_3$. 
As the normals are close to radial thanks to \eqref{eq:Rcloseto1}, the three intersection points $L_1\cap L_2$, $L_1\cap L_3$ and $L_2\cap L_3$ lie in $D_{c_0\e_0}$ for some absolute constant $c_0$. 
Recall that $z_\alpha$ is the intersection point $z_\alpha =L_2\cap L_3$, and denote by $d$ its distance to the line $L_1$, $d=\dist(z_\alpha,L_1)\lesssim \e_0$. 
Since $x_1\in I_\alpha$, $\bar x_2\in I_{\alpha +2\pi/3}$ and $\bar x_3\in I_{\alpha+4\pi/3}$, the triangle formed inside $D_{c_0\e_0}$ by the three lines $L_1,L_2,L_3$ has its three angles $\gtrsim 1$ for small enough $\e_0$. 
Hence the radius $r$ of that triangle's incircle is comparable to the distance $d$, we have $d\lesssim r$. 
On the other hand, considering the three concurrent lines from $x_1,\bar x_2,\bar x_3$ to the incircle's center, we find that $r\lesssim a(x_1,\bar x_2,\bar x_3)$. 
Thus we have $d\lesssim r\lesssim a(x_1,\bar x_2,\bar x_3)$. Moreover, the angle between $L_1$ and the line from $z_\alpha$ to $x_1$ is $\lesssim d$, which shows that 
\begin{align*}
\left|n_{\partial\Omega}(x_1)-n_\alpha(x_1)\right|
\lesssim d\lesssim r\lesssim a(x_1,\bar x_2,\bar x_3),
\end{align*}
and proves the claim \eqref{eq:aboundsnormal}.
From \eqref{eq:aboundsnormal} we deduce
\begin{align*}
\int_{I_\alpha} \left| n_{\partial\Omega}-n_\alpha\right|^2 d\mathcal H^1
\lesssim
\int_{I_\alpha} a^2(\cdot,\bar x_2,\bar x_3)\,d\mathcal H^1 
\lesssim
\int_{\partial\Omega^3}a^2 \, d(\mathcal H^1)^{\otimes 3}.
\end{align*}
Applying this to $\alpha=j\pi/12$ for $j=1,\ldots,24$, we cover $\partial\Omega$ with portions $I_{1},\ldots,I_{24}$ satisfying $\mathcal H^1(I_{j}\cap I_{j+1})=\pi/12 +\mathcal O(\e_0)$, and  find points $z_j\in D_{c_0\e_0}$  
such that
\begin{align}\label{eq:nIj}
\int_{I_j} \left| n_{\partial\Omega}-n_j\right|^2 d\mathcal H^1
\lesssim
\int_{\partial\Omega^3}a^2 \, d(\mathcal H^1)^{\otimes 3},\qquad 
n_j(x)=\frac{x-z_j}{|x-z_j|}.
\end{align}
This implies
\begin{align*}
\int_{I_j\cap I_{j+1}}|n_j-n_{j+1}|^2 d\mathcal H^1
\lesssim \int_{\partial\Omega^3}a^2 \, d(\mathcal H^1)^{\otimes 3}.
\end{align*}
We claim that
\begin{align}\label{eq:boundzj}
|z_j -z_{j+1}|^2 \lesssim 
\int_{I_j\cap I_{j+1}}|n_j-n_{j+1}|^2 d\mathcal H^1 \lesssim \int_{\partial\Omega^3}a^2 \, d(\mathcal H^1)^{\otimes 3}.
\end{align}
The second inequality was proved above, so we only need to show the first inequality in \eqref{eq:boundzj}. First note that since $z_j,z_{j+1}\in D_{c_0\e_0}$ and $I_j\cap I_{j+1}\subset D_{1+c_0\e_0}\setminus D_{1-c_0\e_0}$, for any $x\in I_j\cap I_{j+1}$ we have
\begin{align*}
|(z_{j+1}-z_j)\cdot (in_j(x))| \lesssim |n_j(x)-n_{j+1}(x)|.
\end{align*}
In other words, $|n_j(x)-n_{j+1}(x)|$ controls $|z_j-z_{j+1}|$ unless $x,z_j,z_{j+1}$ are closed to aligned. But since $I_{j}\cap I_{j+1}$ is a portion of curve inside $ D_{1+c_0\e_0}\setminus D_{1-c_0\e_0}$ from a point of polar angle $(j+1)\pi/12$ to a point of polar angle $(j+2)\pi/12$, there is a subset $J\subset I_{j}\cap I_{j+1}$ satisfying $\mathcal H^1(J)\geq \pi/24$ and such that for $x\in J$ the three points $x,z_j,z_{j+1}$ are far from aligned, that is,
\begin{align*}
|z_j-z_{j+1}|\lesssim |(z_{j+1}-z_j)\cdot (in_j(x))| \lesssim |n_j(x)-n_{j+1}(x)|\qquad\forall x\in J.
\end{align*}
Taking squares and integrating over $J$ we obtain \eqref{eq:boundzj}.
From \eqref{eq:boundzj}
we deduce $|z_j-z_1|^2\lesssim  \int_{\partial\Omega^3}a^2 \, d(\mathcal H^1)^{\otimes 3}$ for all $j=1,\ldots,24$, and therefore \eqref{eq:nIj} implies
\begin{align*}
\int_{\partial\Omega} |n_{\partial\Omega}-n_1|^2 \lesssim 
 \int_{\partial\Omega^3}a^2 \, d(\mathcal H^1)^{\otimes 3}.
\end{align*}
Taking $x_*=z_1$ and applying Lemma~\ref{l:aux} and Cauchy-Schwarz' inequality, this concludes the proof of Proposition~\ref{p:est_a_circle}.
\end{proof}

\section{Lower bound on the dissipation}\label{s:dissip}

In this section we prove the following.

\begin{prop}\label{p:adissiptrace}
Let $\Omega$ and $m$ be as in Theorem \ref{t:main2}. We have the estimate
\begin{align*}
\int_{E(\eta_*)^3} a^2 d(\mathcal H^1)^{\otimes 3} \leq C\,  \nu(\Omega) + C\,  \mathcal H^1(\lbrace m_{\partial\Omega}=\tau\rbrace),
\end{align*}
where $\eta_*=\min(\eta_0/2,1/(8K))$, for $\eta_0$ as in Lemma~\ref{l:segment}, and $C>0$ is a constant depending only on $K$.
\end{prop}

\subsection{Lagrangian representation}\label{ss:lag}

In order to prove Proposition  \ref{p:adissiptrace}, we introduce the notion of Lagrangian representation for entropy solutions of the eikonal equation from \cite{marconi21ellipse}.

Given $T>0$ we let
\begin{align*}
\Gamma= \Big\{ (\gamma,t^-_\gamma,t^+_\gamma)\colon
& 0\le t^-_\gamma\le t^+_\gamma\le T, \\
&
\gamma=(\gamma_x,\gamma_s)\in \BV((t^-_\gamma,t^+_\gamma);\Omega \times \R/2\pi \Z),
 \gamma_x \mbox{ is Lipschitz} \Big\}.
\end{align*}
We will always consider the right-continuous representative of the component $\gamma_s$ and we will write $\gamma(t^-_\gamma)$ instead of $ \lim_{t\to t^-_\gamma} \gamma(t)$ and  $\gamma(t^+_\gamma)$ instead of $ \lim_{t\to t^+_\gamma} \gamma(t)$.
For every $t \in (0,T)$ we consider the section 
\begin{equation*}
 \Gamma(t):= \left\{\left(\gamma,t^-_\gamma,t^+_\gamma\right)\in  \Gamma: t \in \left(t^-_\gamma,t^+_\gamma\right)\right\}
\end{equation*}
and we denote by 
\begin{align*}
\begin{array}{crl}
 e_t\colon &
  \Gamma(t) &\longrightarrow \Omega \times \R/2\pi \Z \\
&(\gamma, t^-_\gamma, t^+_\gamma)  &\longmapsto  \gamma(t),
\end{array}
\end{align*}
the evaluation map at time $t$.

\begin{definition}\label{D_Lagr}
Let $\Omega$ be a $C^{1,1}$ open set and $m$ solving \eqref{eq:eik} and \eqref{eq:kin}.
We say that a finite non-negative Radon measure $\omega \in \mathcal M( \Gamma)$ is a \emph{Lagrangian representation} of $m$ if the following conditions are satisfied:
\begin{enumerate}
\item for every $t\in (0,T)$ we have
\begin{equation}\label{E_repr_formula}
( e_t)_\sharp \left[ \omega \llcorner  \Gamma(t)\right]= \mathbf 1_{E_m}\, \L^{2}\times \L^1,
\end{equation}
where $E_m\subset\Omega\times\R/2\pi\Z$ is the `epigraph'
\begin{align*}
E_m=\left\lbrace (x,s)\in\Omega\times\R/2\pi\Z\colon m(x)\cdot e^{is}> 0\right\rbrace;
\end{align*}
\item the measure $\omega$ is concentrated on curves $(\gamma,t^-_\gamma,t^+_\gamma)\in  \Gamma$ solving the characteristic equation:
\begin{equation}\label{e:charac}
\dot\gamma_x(t)= e^{i \gamma_s(t)}\qquad\text{for a.e. }t\in (t^-_\gamma,t^+_\gamma);
\end{equation}
\item we have the integral bound
\begin{equation*}\label{E_reg}
\int_{ \Gamma} \TV_{(0,T)} \gamma_s d\omega(\gamma) <\infty;
\end{equation*}
\item for $\omega$-a.e. $(\gamma,t^-_\gamma,t^+_\gamma)\in \Gamma$ we have
\begin{equation}\label{e:extreme_time}
t^-_\gamma>0 \Rightarrow  \gamma_x(t^-_\gamma ) \in \partial \Omega, \qquad \mbox{and} \qquad 
t^+_\gamma<T \Rightarrow  \gamma_x(t^+_\gamma) \in \partial \Omega.
\end{equation}
\end{enumerate}
\end{definition}

A useful property of the Lagrangian formulation is the possibility of decomposing the entropy dissipation measure $\nu$ along the characteristics detected by $\omega$.
More precisely, from \cite{marconi21ellipse} we have

\begin{prop}\label{p:lagrangiandissipation}
Let $\Omega$ be a $C^{1,1}$ open set, $m$ solving \eqref{eq:eik} and \eqref{eq:kin}, and $T>0$. Then there is a Lagrangian representation $\omega$ of $m$ such that for every Borel set $A\subset [0,T]$ and $B\subset \Omega$ it holds
\[
\int_\Gamma \mu_\gamma (\{ t \in A : \gamma_x(t) \in B\})\, d\omega(\gamma) = \L^1(A) \nu(B),
\]
where $\mu_\gamma =|D_t\gamma_s|$.
\end{prop}

\begin{rem}\label{r:mugamma}
Note that $\gamma_s$ takes values into $\R/2\pi\Z$, 
so a few precisions about the meaning of the measure $\mu_\gamma =|D_t\gamma_s|$ are in order.
It should actually be understood as the measure $|D_t\hat\gamma_s|$ where $\hat\gamma_s\in BV(I_\gamma;\R)$ is such that $\gamma_s(t)=\hat\gamma_s(t)+2\pi\Z$ for all $t\in I_\gamma$, and the jumps of $\hat\gamma_s$ are such that $|\hat\gamma_s(t+)-\hat\gamma_s(t-)|=\dist_{\R/2\pi\Z}(\gamma_s(t-),\gamma_s(t+))$ (see e.g. \cite[Theorem~1]{ignat05} for the existence of such a lifting, which is however not necessary to define the measure $\mu_\gamma$). 
\end{rem}

We will also use that, thanks to property \eqref{e:extreme_time} and the trace properties of $m$, the pushforward of $\omega$ under evaluation at initial time $t_\gamma^-$ is related to the $\mathcal H^1_{\lfloor\partial\Omega}$ in the following way.
\begin{lem}\label{l:projinit}
Denote $\Gamma_{ini}=\lbrace t_\gamma^->0\rbrace\subset\Gamma$ and 
\begin{align*}
P_{ini}\colon \Gamma_{ini}\to (0,T)\times \partial\Omega\times\R/2\pi\Z,\quad \gamma\mapsto (t_\gamma^-,\gamma(t_\gamma^-)),
\end{align*}
then the pushforward measure $\mu_{ini}=P_{ini}\sharp \omega_{\lfloor\Gamma_{ini}}$ is given by
\begin{align*}
d\mu_{ini}(t,x,s) 
&=\mathbf 1_{m (x)\cdot e^{is}>0} \mathbf 1_{i\tau(x)\cdot e^{is}>0} \,( i\tau(x)\cdot e^{is}) \, dt \, d\mathcal H^1_{\lfloor\partial\Omega}(x) \, ds.
\end{align*}
\end{lem}

\begin{proof}[Proof of Lemma~\ref{l:projinit}]
The argument is similar to \cite[Lemma~3.1]{CHLM22}.
Let $F\in C^1_c((0,T)\times \partial\Omega\times\R/2\pi\Z)$, and denote also by $F$ a $C^1$ extension to $(0,T)\times \overline\Omega\times\R/2\pi\Z$. 

For small enough $h>0$ we may find a $C/h$-Lipschitz function $G_h\colon\R^2\to [0,1]$, with $C$ depending only on $K$, such that
\begin{align*}
&\mathbf 1_{x\in\Omega,\, \dist(x,\partial\Omega)\leq h} \leq G_h(x) \leq \mathbf 1_{x\in\Omega},\\
\text{and }&\nabla G_h \to i\tau \, d\mathcal H^1_{\lfloor \partial\Omega}\qquad\text{as }h\to 0.
\end{align*}
Thanks to the trace property of $m$ and the Lagrangian property \eqref{E_repr_formula} we have
\begin{align*}
&\int_{(0,T)\times \partial\Omega\times\R/2\pi\Z}
 F(t,x,s) \,\mathbf 1_{m (x)\cdot e^{is}>0} \, 
( i\tau(x)\cdot e^{is}) \, 
dt \, d\mathcal H^1_{\lfloor\partial\Omega}(x) \, ds \\
&=\lim_{h\to 0^+}
\int_{(0,T)\times \Omega\times\R/2\pi\Z} F(t,x,s) \,
\mathbf 1_{m(x)\cdot e^{is}>0} \,
(e^{is}\cdot\nabla G_h(x))
\, dt \, dx \, ds \\
&=\lim_{h\to 0^+}
\int_{\Gamma}
 A_h(\gamma) \, d\omega(\gamma), \\
\end{align*}
where
\begin{align*}
A_h(\gamma)&=\int_{t_\gamma^-}^{t_\gamma^+}  F(t,\gamma_x(t),\gamma_s(t)) \,
 \, ( e^{i\gamma_s(t)} \cdot \nabla G_h(\gamma_x(t))
 \, dt \\
 &=\int_{t_\gamma^-}^{t_\gamma^+}  F(t,\gamma_x(t),\gamma_s(t)) \,
 \, \frac{d}{dt}\left[ G_h(\gamma_x(t))\right]
 \, dt.
\end{align*}
For the last equality we used the characteristic equation \eqref{e:charac}. Then we integrate by parts: since $G_h(\gamma_x(t_\gamma^-))=0$ if $t_\gamma^->0$ and $G_h(t_\gamma^+)=0$ if $t_\gamma^+<T$ we obtain 
\begin{align*}
A_h(\gamma)&=-\int_{t_\gamma^-}^{t_\gamma^+}  G_h(\gamma_x(t)) D\Phi_\gamma(dt),\\
\text{where }\Phi_\gamma(t)&=F(t,\gamma_x(t),\gamma_s(t)).
\end{align*}
In particular we have  the convergence 
\begin{align*}
A_h(\gamma)\longrightarrow A_0(\gamma)=\int_{t_\gamma^-}^{t_\gamma^+} \mathbf 1_{\gamma_x(t)\in\Omega} D\Phi_\gamma(dt),
\end{align*}
as $h\to 0^+$.  By definition of the Lagrangian representation, for $\omega$-a.e. $\gamma\in\Gamma$ we have $\gamma_x(t)\in\Omega$ for all $t\in (t_\gamma^-,t_\gamma^+)$, so 
\begin{align*}
A_0(\gamma)&=-\int_{t_\gamma^-}^{t_\gamma^+} D\Phi_\gamma(dt)
=\Phi_\gamma(t_\gamma^-)-\Phi_\gamma(t_\gamma^+) \\
& =F(t_\gamma^-,\gamma_x(t_\gamma^-),\gamma_s(t_\gamma^-)) -  F(t_\gamma^+,\gamma_x(t_\gamma^+),\gamma_s(t_\gamma^+)) .
\end{align*}
Thanks to the domination  $|A_h(\gamma)|\leq \| \nabla F\|_\infty (1+TV(\gamma_s))$,
 we deduce
\begin{align*}
&\int_{(0,T)\times \partial\Omega\times\R/2\pi\Z} F(t,x,s) \,
\mathbf 1_{m (x)\cdot e^{is}>0} \, 
( i\tau(x)\cdot e^{is}) \, 
dt \, d\mathcal H^1_{\lfloor\partial\Omega}(x) \, ds \\
&  =\int_{\Gamma} \left( F(t_\gamma^-,\gamma_x(t_\gamma^-),\gamma_s(t_\gamma^-)) -  F(t_\gamma^+,\gamma_x(t_\gamma^+),\gamma_s(t_\gamma^+))\right)\, d\omega(\gamma).
\end{align*}
We apply this to
\begin{align*}
F(t,x,s)=f(t,x,s)\phi_\e(x,s),\quad
\mathbf 1_{i\tau(x)\cdot e^{is}>\e}\leq \phi_\e(x,s)\leq \mathbf 1_{i\tau(x)\cdot e^{is}>0},
\end{align*}
where $f\in C_c^1((0,T)\times \partial\Omega\times\R/2\pi\Z )$ and $\phi_\e\in C^1( \partial\Omega\times\R/2\pi\Z)$. 
Since for $\omega$-a.e. $\gamma\in\Gamma$ it holds $\gamma_x(t)\in\Omega$ for all $t\in (t_\gamma^-,t_\gamma^+)$, 
by the characteristic speed constraint \eqref{e:charac}, we have $i\tau(\gamma_x(t_\gamma^+))\cdot e^{i\gamma_s(t_\gamma^+)}\leq 0$ if $t_\gamma^+< T$, so $\phi_\e(\gamma_x(t_\gamma^+),\gamma_s(t_\gamma^+))=0$, and we deduce
\begin{align*}
&\int_{(0,T)\times \partial\Omega\times\R/2\pi\Z} f(t,x,s)\phi_\e(x,s) \,
\mathbf 1_{m (x)\cdot e^{is}>0} \, 
( i\tau(x)\cdot e^{is}) \, 
dt \, d\mathcal H^1_{\lfloor\partial\Omega}(x) \, ds \\
&  =\int_{\Gamma} f(t_\gamma^-,\gamma_x(t_\gamma^-),\gamma_s(t_\gamma^-)) \phi_\e(\gamma_x(t_\gamma^-),\gamma_s(t_\gamma^-))\, d\omega(\gamma).
\end{align*}
By dominated convergence as $\e\to 0$ this implies
\begin{align*}
&\int_{(0,T)\times \partial\Omega\times\R/2\pi\Z} f(t,x,s)
\,\mathbf 1_{i\tau(x)\cdot e^{is}>0} 
\,\mathbf 1_{m (x)\cdot e^{is}>0} \, 
( i\tau(x)\cdot e^{is}) \, 
dt \, d\mathcal H^1_{\lfloor\partial\Omega}(x) \, ds \\
&  =\int_{\Gamma} f(t_\gamma^-,\gamma_x(t_\gamma^-),\gamma_s(t_\gamma^-)) 
\,\mathbf 1_{i\tau(\gamma_x(t_\gamma^-))\cdot e^{i\gamma_s(t_\gamma^-))}>0} 
\, d\omega(\gamma).
\end{align*}
As above, for $\omega$-a.e. $\gamma\in\Gamma_{ini}$ we have  $i\tau(\gamma_x(t_\gamma^-))\cdot e^{i\gamma_s(t_\gamma^-))}\ge0$.
Indeed the curves that enter tangentially into $\Omega$, namely for which $i\tau(\gamma_x(t_\gamma^-))\cdot e^{i\gamma_s(t_\gamma^-))} = 0$, 
are $\omega$-negligible (see \cite[(3.5)]{CHLM22} for details).
In particular for $\omega$-a.e. $\gamma\in\Gamma_{ini}$ we have  $i\tau(\gamma_x(t_\gamma^-))\cdot e^{i\gamma_s(t_\gamma^-))}>0$ and we infer
\begin{align*}
&\int_{(0,T)\times \partial\Omega\times\R/2\pi\Z} f(t,x,s)
\,\mathbf 1_{i\tau(x)\cdot e^{is}>0} 
\,\mathbf 1_{m (x)\cdot e^{is}>0} \, 
( i\tau(x)\cdot e^{is}) \, 
dt \, d\mathcal H^1_{\lfloor\partial\Omega}(x) \, ds \\
&  =\int_{\Gamma_{ini}} f(t_\gamma^-,\gamma_x(t_\gamma^-),\gamma_s(t_\gamma^-))  
\, d\omega(\gamma),
\end{align*}
for any $f\in C_c^1((0,T)\times \partial\Omega\times\R/2\pi\Z)$. By approximation this is valid for any $f\in C_c^0(0,T)\times \partial\Omega\times\R/2\pi\Z)$, concluding the proof of Lemma~\ref{l:projinit}.
\end{proof}

\subsection{Proof of Proposition~\ref{p:adissiptrace}}

Before proving Proposition~\ref{p:adissiptrace} we set some notations and definitions.
\begin{itemize}
\item We apply Proposition~\ref{p:lagrangiandissipation} for some fixed $T\geq 3\pi$, and let $h \in L^1(\partial \Omega)$  be defined by the relation 
\begin{align}\label{eq:h}
\int_E h\, d\mathcal H^1 = \int_{\{\gamma : \gamma_x (t^-_\gamma) \in E\}} |D_t \gamma_s| (I_\gamma) \, d\omega_h(\gamma),
\end{align}
that is, $h(x)$ encodes the entropy dissipation generated (via Proposition~\ref{p:lagrangiandissipation}) by the curves of the Lagrangian representation emanating from $x\in\partial\Omega$.

\item We denote by $W$ the `wrong trace' set $W=\lbrace m=\tau\rbrace\subset\partial\Omega$,  by $\mathcal M \mathbf 1_W$ the maximal function
\begin{align*}
\mathcal M\mathbf 1_W (x)=\sup_{r>0}\frac{1}{r}\int_{I_r(x)} \mathbf 1_W d\mathcal H^1,
\end{align*}
where $I_r(x)=\p([t-r,t+r])$ for $x=\p(t)$,
and for any $\e>0$ we define the set
\begin{align}\label{eq:Geps}
G_\e =\lbrace \mathcal M\mathbf 1_W <\e \rbrace \subset\partial\Omega,
\end{align}
of boundary points where the proportion of wrong traces is less than $\e$ at any scale around that point. Note that the Hardy-Littlewood inequality ensures that $\mathcal H^1(\partial\Omega\setminus G_\e)\lesssim \e^{-1}\mathcal H^1(W)$.
\end{itemize}

The main, and most technical, part of Proposition~\ref{p:adissiptrace}'s proof is encoded in the following lemma.
\begin{lem}\label{l:local_lower_bound}
There exist $C,c,\e>0$, depending only on $K$, such that, for any $\hat x\in (E(\eta_*)\cap G_\e)^3$ with $a(\hat x)>0$, 
the quantity $a(\hat x)$ provides the following lower bound on the entropy dissipation:
\begin{align*}
\frac{C}{T}\int_{I(\hat x, a(\hat x))} (h(x_1) + h(x_2) + h(x_3) ) \,d (\mathcal H^1)^{\otimes 3} \geq  a(\hat x)^5,
\end{align*}
where
\begin{align*}
I(\hat x, a(\hat x))& = I_1 \times I_2 \times I_3 ,\qquad
I_k = I_{c\,a(\hat x)}(x_k)=\p([\bar s_k-c\, a(\hat x),\bar s_k +c\, a(\hat x)]),
\end{align*}
and $x_k=\p(\bar x_k)$ for $k=1,2,3$.
\end{lem}

\begin{proof}[Proof of Lemma~\ref{l:local_lower_bound}]
Let $\hat x \in (E(\eta_*)\cap G_\e)^3$ with $a(\hat x)>0$

First we choose the constant $c=c(K)>0$ appearing in the definition of $I(\hat x,a(\hat x))$ in order to ensure
\begin{align*}
e^{i\alpha_k}\cdot \tau(x)\leq -\frac{a(\hat x)}{2}\;\forall x\in I_k,
\quad
\text{and }
I_k\subset E(2\eta_*),
\end{align*}
where $\alpha_k$ are the angles in the definition of $a(\hat x)$. 
The first condition can be imposed because $\tau$ is $K$-Lipschitz and $e^{i\alpha_k}\cdot \tau(x_k)\leq -a(\hat x)$, and the second thanks to Lemma~\ref{l:continuation}, since $2\eta_*\leq 1/(4K)$ (as imposed in the statement of Proposition~\ref{p:adissiptrace}).

We denote by $z_0$ the intersection point as in the definition of $a(\hat x)$, and consider the three cylinders
\[
\mathcal C_k:=  B_{c_1a(\hat x)} (z_0) \times [\alpha_k - c_2 a(\hat x), \alpha_k + c_2 a(\hat x)] \qquad \mbox{for } k=1,2,3,
\]
where $c_1,c_2>0$ are small constants depending on $K$ and chosen to ensure that:
\begin{itemize}
\item  for any $(s_1,s_2,s_3)\in\Pi_{k=1}^3 [\alpha_k - c_2 a(\hat x), \alpha_k + c_2 a(\hat x)]$, the shortest interval in $\R/2\pi\Z$ containing $s_1,s_2,s_3$ has length $l > \pi$;
\item for all $(x,s)\in\mathcal C_k$, there is a  boundary point $y\in I_k $ such that $x=y+te^{is}$ for some $t\in\R$, the segment $[x,y]$ is contained in $\overline\Omega$, and $\tau(y)\cdot e^{is}<0$. 
\end{itemize}
The last property is possible   since $2\eta_*\leq \eta_0 $ with $\eta_0$ as in Lemma~\ref{l:segment}.

\medskip

\noindent\textbf{Claim.} For every $k=1,2,3$ we have at least one of the following two properties:
\begin{align*}
\mathcal L^3(\mathcal C_k \cap E_m) \ge \frac34 \mathcal L^3 (\mathcal C_k),\qquad
\text{or }\quad &\frac1T\int_{I_k}h\, d\mathcal H^1 \gtrsim a(\hat x)^3.
\end{align*}
In other words, either most of the elements of $\mathcal C_k$ belong to the `epigraph' $E_m$, or there must occur entropy dissipation of order $a(\hat x)$.

\medskip 

We first prove the statement assuming the Claim.
Since $\mathcal H^1(I_k)= 2c\,a(\hat x) $ for every $k=1,2,3$, then
\[
\int_{I(\hat x, a(\hat x))} (h(x_1) + h(x_2) + h(x_3) ) \,d (\mathcal H^1)^{\otimes 3} =2c\,  a(\hat x)^2 \sum_{k=1}^3 \int_{I_k}h \, d\mathcal H^1 .
\]
In view of the Claim, in order to conclude the proof it is sufficient to check that
the first property $\mathcal L^3(\mathcal C_k \cap E_m) \ge (3/4) \mathcal L^3 (\mathcal C_k)$ cannot be satisfied for all $k=1,2,3$.
Assume by contradiction that this is the case:
for every $k=1,2,3$, let 
\[
A_k= \{ x \in B_{c_1a(\hat x)}(z_0): \mathcal L^1(\{ s: (x,s) \in \mathcal C_k \cap E_m\})>0\}.
\]
In particular we have  $\mathcal L^2(A_k)\ge \frac34 \mathcal L^2( B_{c_1a(\hat x)}(z_0))$ so that 
\[
\mathcal L^2(A_1\cap A_2 \cap A_3) \ge \frac14 \mathcal L^2( B_{c_1a(\hat x)}(z_0))>0.
\]
But the definition of $a(\hat x)$ implies that $A_1\cap A_2\cap A_3=\emptyset$. Indeed for every triple $(s_1,s_2,s_3) \in \prod_{k=1}^3[\alpha_k - c_2 a(\hat x), \alpha_k + c_2 a(\hat x)]$,
on the one hand the choice of $c_2$ ensures that
there is no $\alpha \in \R/2\pi\Z$ such that $e^{i\alpha}\cdot e^{is_k}>0$ for every $k=1,2,3$, and on the other hand for $\mathcal L^3$-a.e. $(x,s) \in E_m$ we have
$m(x)\cdot e^{is}>0$.  So this gives a contradiction.

\medskip

It remains to prove the Claim. 
For $k=1,2,3$, we consider the set of curves $G_k\subset \Gamma$ defined as follows.
If the direction $e^{i\alpha_k}$ enters $\Omega$, $G_k$ consists of the curves which enter $\Omega$
in a way that the  `free characteristic' (i.e. straight line) entering with the same initial  direction intersects the cylinder $\mathcal C_k$. 
 If the direction $e^{i\alpha_k}$ exits $\Omega$, $G_k$ consists of the curves which exit $\Omega$
in a way that the free characteristic exiting with the same final  direction intersects the cylinder $\mathcal C_k$.
Explicitly:
\begin{align*}
 G_k  =\Bigg\{ \gamma \in \Gamma \colon  \exists \bar t \in \left[\frac{T}3,\frac23T\right], 
& 
 \left( \gamma_x(t^-_\gamma) + e^{i\gamma_s(t^-_\gamma)}(\bar t - t^-_\gamma),\gamma_s(t_\gamma^-)\right) \in   \mathcal C_k \Bigg\} \\
 & \hspace{10em}\text{if }(x_k - z_0) \cdot \tau (x_k) >0, \\
 G_k=\Bigg\{ \gamma \in \Gamma \colon   \exists \bar t \in \left[\frac{T}3,\frac23T\right],
 &  
 \left(\gamma_x(t^+_\gamma) + e^{i\gamma_s(t^+_\gamma)}(\bar t - t^+_\gamma),\gamma_s(t_\gamma^+)\right) \in  \mathcal C_k \Bigg\} \\
 &\hspace{10em} \text{if }(x_k - z_0) \cdot \tau (x_k) <0.
\end{align*}
Moreover for $\gamma\in G_k$ we denote by $t_{\mathcal C_k}(\gamma)$ the time spent by $\gamma$ in $\mathcal C_k$, and by $\tilde t_{\mathcal C_k}(\gamma)$ the time spent in $\mathcal C_k$ by the corresponding (entering or exiting) free characteristic.
Explicitly:
\begin{align*}
t_{\mathcal C_k}(\gamma) & = \mathcal L^1\left(\left\{ t \in \left[\frac{T}3,\frac23T\right]:\gamma(t) \in \mathcal C_k \right\}\right) & \\
\tilde t_{\mathcal C_k}(\gamma)& = \mathcal L^1\left(\left\{ t \in \left[\frac{T}3,\frac23T\right],
\left( \gamma_x(t^-_\gamma) + e^{i\gamma_s(t^-_\gamma)}(  t - t^-_\gamma),\gamma_s(t_\gamma^-)\right) \in  \mathcal C_k \right\} \right) \\
&\hspace{20em} \mbox{if }(x_k - z_0) \cdot \tau (x_k) >0, \\
\tilde t_{\mathcal C_k}(\gamma)& = \mathcal L^1\left(\left\{ t \in \left[\frac{T}3,\frac23T\right],
\left(  \gamma_x(t^+_\gamma) + e^{i\gamma_s(t^+_\gamma)}(\bar t - t^+_\gamma),\gamma_s(t_\gamma^+)\right) \in  \mathcal C_k \right\} \right) \\
& \hspace{20em}\mbox{if }(x_k - z_0) \cdot \tau (x_k) <0,
\end{align*}

Since $T\ge 3\pi$, the choices of $c_1,c_2$ ensure that
\begin{align}\label{e:tildet}
&(1-\beta)\frac{T}3 \mathcal L^3(\mathcal C_k)
\leq
\int_{G_k} \tilde t_{\mathcal C_k} d \omega \leq \frac{T}3 \mathcal L^3(\mathcal C_k),\\
\text{where }&
\beta\lesssim \frac{\mathcal H^1(\lbrace x\in I_k(\hat x,a(\hat x))\colon m(x)=\tau(x)\rbrace)}{\mathcal H^1(I_k(\hat x,a(\hat x)))}\lesssim\e.
\nonumber
\end{align}
To prove \eqref{e:tildet}, assume without loss of generality that $(x_k-z_0)\cdot \tau(x_k)>0$. Then, by Fubini theorem, we have
\begin{align*}
\int_{G_k} \tilde t_{\mathcal C_k} d \omega &
=\int_{T/3}^{2T/3}\omega\left(A_t\right)\, dt, \\
\text{where}\quad A_t &= 
\left\lbrace \gamma\in \Gamma \colon
\left( \gamma_x(t^-_\gamma) + e^{i\gamma_s(t^-_\gamma)}(  t - t^-_\gamma),\gamma_s(t_\gamma^-)\right) \in  \mathcal C_k\right\rbrace.
\end{align*}
Since $T/3\geq \pi$ and $\Omega$ has diameter $<\pi$, for $t\in [T/3,2T/3]$ any $\gamma\in A_t$ satisfies $t_\gamma^->0$, and moreover $\gamma_x(t_\gamma^-)\in I_k$ 
and $i\tau(\gamma_x(t_\gamma^-))\cdot e^{i\gamma_s(t_\gamma^-)}>0$
thanks to the choices of $c_1,c_2$.
Invoking Lemma~\ref{l:projinit} this implies that
\begin{align*}
\omega(A_t)  = 
\int_{\alpha_k-c_2a(\hat x)}^{\alpha_k+c_2a(\hat x)}
\int_{I_k}
\int_0^{2T/3}
&
\mathbf 1_{y+(t-t_{ini})e^{is} \in B_{c_1a(\hat x)}(z_0)} 
\\
&\cdot 
 \mathbf 1_{m(y)\cdot e^{is}>0}( i\tau(y)\cdot e^{is})\, 
 dt_{ini}\,d\mathcal H^1(y)\, ds.
\end{align*}
For any $s\in [\alpha_k-c_2a(\hat x),\alpha_k+c_2a(\hat x)]$ the map $(t_{ini},y)\mapsto y +(t-t_{ini})e^{is}$ is injective, its image contains $ B_{c_1a(\hat x)}(z_0)$, and its jacobian is $i\tau(y)\cdot e^{is}>0$, 
so we deduce 
\begin{align*}
\omega(A_t)  = 
\int_{\alpha_k-c_2a(\hat x)}^{\alpha_k+c_2a(\hat x)}
\int_{B_{c_1a(\hat x)}(z_0)}
\mathbf 1_{m(x_s(z))\cdot e^{is}>0}
\,
 dz\, ds,
\end{align*}
where $x_s(z)\in I_k$ is the intersection point of the half-line $z- [0,\infty)\, e^{is}$ with the boundary arc $I_k$. For $z\in B_{c_1 a(\hat x)}(z_0)$ and $s\in [\alpha_k-c_1 a(\hat x),\alpha_k+c_2 a(\hat x)]$, recalling that $m\in\lbrace \pm\tau\}$ on $\partial\Omega$, we have that $m(x_s(z))\cdot e^{is}>0$ if and only if $m(x_s(z))=-\tau(x_s(z))$, and we deduce the validity of \eqref{e:tildet} with
\begin{align*}
\beta= \sup_{z\in B_{c_1 a(\hat x)}(z_0)} \frac{ \mathcal L^1(\lbrace s\in [\alpha_k-c_1 a(\hat x),\alpha_k+c_2 a(\hat x)] \colon m(x_s(z))=\tau(x_s(z))\rbrace)}
{\mathcal L^1(  [\alpha_k-c_1 a(\hat x),\alpha_k+c_2 a(\hat x)] )}.
\end{align*}
The first inequality on $\beta$ in the second line of \eqref{e:tildet} follows from the fact that, for all  $z\in B_{c_1 a(\hat x)}(z_0)$, the map $s\mapsto x_s(z)$ is a diffeomorphism from $[\alpha_k-c_1 a(\hat x),\alpha_k+c_2 a(\hat x)]$ onto its image in $I_k$, with jacobian bounded from below by $R/2\geq 1/(2K)$. 
The second inequality $\beta\lesssim\e$ in \eqref{e:tildet} is simply by definition \eqref{eq:Geps} of the set $G_\e$.

Moreover, since for every $t \in \left[0,T\right]$ we have $(e_t)_\# \omega =\mathbf 1_{E_m} \mathcal L^3$, then
\begin{equation}\label{e:t}
\int_{G_k}t_{\mathcal C_k} d\omega \le \frac{T}3\mathcal L^3(\mathcal C_k \cap E_m).
\end{equation}
We now estimate $\tilde t_{\mathcal C_k} - t_{\mathcal C_k}$ in terms of the entropy dissipation from the curves in $G_k$.
Assume without loss of generality that $k$ is such that $(x_k - z_0) \cdot \tau (x_k) >0$.

Denote by 
\[
\tilde {\mathcal C}_k = \{ (x,s) \in \mathcal C_k : \dist( (x,s), \partial \mathcal C_k)> c_4 a(\hat x) \}
\]
for some $c_4 \in (0, \min\{ c_1, c_2/2\})$
and by $\tilde G_k \subset G_k$ the set of curves $\gamma$ such that there is $t\in (t^-_\gamma, t^+_\gamma)$ such that 
$\gamma_x(t^-_\gamma) + e^{i\gamma_s(t^-_\gamma)}( t - t^-_\gamma) \in \tilde{\mathcal C}_k$. 
Finally denote by 
\[
\tilde G_k^* = \{ \gamma \in \tilde G_k : \tilde t_{\mathcal C_k}(\gamma) - t_{\mathcal C_k}(\gamma) > c_5 a(\hat x)\}.
\]

For every $c_4,c_5>0$, by the characteristic constraint \eqref{e:charac}, there is $c_6>0$ such that for all $\gamma \in \tilde G_k^*$ it holds 
$|\mu_\gamma|((0,T))\ge c_6 a(\hat x)$.

We write 
\begin{align}\label{eq:inttildet}
\int_{G_k} (\tilde t_{\mathcal C_k} - t_{\mathcal C_k}) d\omega 
&
\le 
  \int_{G_k\setminus \tilde G_k}\tilde t_{\mathcal C_k} d\omega + \int_{\tilde G_k\setminus \tilde G_k^*}  (\tilde t_{\mathcal C_k} - t_{\mathcal C_k}) d\omega + \int_{\tilde G_k^*}  \tilde t_{\mathcal C_k}  d\omega,
\end{align}
and estimate each integral separately. First, the argument leading to the upper bound in \eqref{e:tildet} also implies
\begin{align}\label{eq:inttildet1}
\int_{G_k\setminus \tilde G_k}\tilde t_{\mathcal C_k} d\omega 
& \leq\frac{T}{3}\mathcal L^3(\mathcal C_k\setminus \tilde{\mathcal C}_k) 
\leq \frac{2T}{3}c_4\left( \frac1{c_1}+\frac1{c_2}\right)\mathcal L^3(\mathcal C_k).
\end{align}
Second,  by definition of $\tilde G_k^*$ we have
\begin{align*}
\int_{\tilde G_k\setminus \tilde G_k^*}  (\tilde t_{\mathcal C_k} - t_{\mathcal C_k}) d\omega
\leq c_5 a(\hat x)\omega(\tilde G_k),
\end{align*}
and since $\tilde t_{C_k}(\gamma)\geq \sqrt{c_1c_4}a(\hat x)$ for all $\gamma\in\tilde G_k$, from \eqref{e:tildet} we deduce $\sqrt{c_1c_4}a(\hat x)\omega(\tilde G_k)\leq(T/3)\mathcal L^3(\mathcal C_k)$, and plugging this into the previous equation yields
\begin{align}
\label{eq:inttildet2}
\int_{\tilde G_k\setminus \tilde G_k^*}  (\tilde t_{\mathcal C_k} - t_{\mathcal C_k}) d\omega
\leq \frac{T}{3}\frac{c_5}{\sqrt{c_1c_4}}\mathcal L^3(\mathcal C_k).
\end{align}
Third, by definition of $c_6$, the third integral in \eqref{eq:inttildet} enjoys the estimate
\begin{align*}
\int_{\tilde G_k^*}  \tilde t_{\mathcal C_k}  d\omega  &
\leq \sup \tilde t_{\mathcal C_k}\, \omega(\tilde G_k^*) 
\leq \frac{ 2c_1 a(\hat x)}{c_6a(\hat x)}\int_{\tilde G_k^*} |\mu_\gamma|(0,T) d\omega(\gamma)
\\
&\leq 
\frac{ 2c_1}{c_6}\int_{I_k(\hat x, a(\hat x))}h\, d\mathcal H^1,
\end{align*}
so plugging this and \eqref{eq:inttildet1}-\eqref{eq:inttildet2} into \eqref{eq:inttildet} we infer
\begin{align*}
\int_{G_k} (\tilde t_{\mathcal C_k} - t_{\mathcal C_k}) d\omega 
&
\le 
  \frac{2T}{3}c_4\left( \frac1{c_1}+\frac1{c_2}\right)\mathcal L^3(\mathcal C_k)
  + \frac{T}{3}\frac{c_5}{\sqrt{c_1c_4}}\mathcal L^3(\mathcal C_k)
  \\
  &\quad
  + \frac{ 2c_1}{c_6}\int_{I_k(\hat x, a(\hat x))}h\, d\mathcal H^1
\end{align*}
We may choose $c_4$ and $c_5$ small enough so that 
\begin{align*}
 \frac{2T}{3}c_4\left( \frac1{c_1}+\frac1{c_2}\right)\mathcal L^3(\mathcal C_k)
  + \frac{T}{3}\frac{c_5}{\sqrt{c_1c_4}}\mathcal L^3(\mathcal C_k)
  \leq \frac{T}{24}\mathcal L^3(\mathcal C_k)
\end{align*}
so that by \eqref{e:t}  and \eqref{e:tildet}, we deduce
\begin{align*}
\frac{T}3\mathcal L^3(\mathcal C_k \cap E_m) &\ge 
 \int_{G_k}\tilde t_{\mathcal C_k} d\omega - \int_{G_k}(\tilde t_{\mathcal C_k} -  t_{\mathcal C_k}) d\omega \\
& \ge 
\frac T3 \left( \frac 78 - c\e \right)\mathcal L^3(\mathcal C_k) -  \frac{ 2c_1}{c_6}\int_{I_k(\hat x, a(\hat x))}h \,d\mathcal H^1,
\end{align*}
for some absolute constant $c>0$. Choosing $\e=1/(16c)$  we deduce
\[
\mathcal L^3(\mathcal C_k \cap E_m) \ge \frac34 \mathcal L^3(\mathcal C_k ) + \frac1{16} \mathcal L^3(\mathcal C_k ) - \frac{6c_1}{c_6T}\int_{I_k(\hat x, a(\hat x))}h\, d\mathcal H^1.
\]
This estimate implies the claim: if $\mathcal L^3(\mathcal C_k \cap E_m) \leq \frac34 \mathcal L^3(\mathcal C_k )$ then 
 $\int_{I_k(\hat x, a(\hat x))}h\, d\mathcal H^1\geq (c_6T/(6c_1))\mathcal L^3(\mathcal C_k)/16 =(\pi T c_1 c_2 c_6/48)\, a(\hat x)^3$.
\end{proof}

Now Proposition~\ref{p:adissiptrace} follows from Lemma~\ref{l:local_lower_bound} via a covering argument.

\begin{proof}[Proof of Proposition \ref{p:adissiptrace}]
Denote by $L$ the Lipschitz constant of $a$ from Lemma \ref{l:a_Lip}. 
Let $\e>0$ be fixed as in Lemma~\ref{l:local_lower_bound}, and consider the covering 
\begin{align*}
\{ I(\hat x, a(\hat x)): \hat x \in (E(\eta_0)\cap G_\e)^3 \}
\end{align*}
of the set $X^*:=\{\hat x \in (E(\eta_0)\cap G_\e)^3\colon a(x)>0\}$. Since for every $\hat x \in X^*$, the diameter of $I(\hat x, a(\hat x))$ is $\lesssim a(\hat x)$ and the function $a$ is Lipschitz by Lemma~\ref{l:a_Lip}, we have
\begin{align}\label{e:a5}
\int_{I(\hat x, a(\hat x))} a^2\, d (\mathcal H^1)^{\otimes 3} \lesssim a(\hat x)^2 (\mathcal H^1)^{\otimes 3}(I(\hat x, a(\hat x))) \lesssim  a(\hat x)^5.
\end{align}
By Besicovitch covering theorem, there is a subcovering
\begin{align*}
\{  I(\hat x_i, a(\hat x_i)): i \in I \}
\end{align*}
of $X^*$ with finite overlap. 
By \eqref{e:a5}, Lemma \ref{l:local_lower_bound} and the finite overlap property  we obtain
\begin{align*}
\int_{X^*} a^2\, d (\mathcal H^1)^{\otimes 3} 
&
\le  \sum_{i\in I} \int_{ I(\hat x_i, a(\hat x_i))}a^2  d (\mathcal H^1)^{\otimes 3} 
\lesssim  \sum_{i\in I} a(\hat x_i)^5 \\
& \lesssim  \sum_{i\in I}\int_{ I(\hat x_i, r_+(\hat x_i))} \left( h(x_1) + h(x_2) + h(x_3) \right) d (\mathcal H^1)^{\otimes 3}\\
&\lesssim \int_{\partial\Omega^3}\left( h(x_1) + h(x_2) + h(x_3) \right) d (\mathcal H^1)^{\otimes 3} \lesssim \int_{\partial\Omega} h\, d\mathcal H^1.
\end{align*}
The definition \eqref{eq:h} of $h$ and Proposition~\ref{p:lagrangiandissipation} ensure that $\int_{\partial\Omega}h\,d\mathcal H^1\lesssim \nu(\Omega)$, so we deduce
\begin{align*}
\int_{(E(\eta_0)\setminus G_\e)^3} a^2\, d(\mathcal H^1)^{\otimes 3} \lesssim \nu(\Omega),
\end{align*}
which implies, since $0\leq a \leq \pi$,
\begin{align*}
\int_{E(\eta_0)^3}a^2\, d(\mathcal H^1)^{\otimes 3} \lesssim \nu(\Omega) +\mathcal H^1(\partial\Omega\setminus G_\e).
\end{align*}
Recalling the definition \eqref{eq:Geps} of $G_\e$, thanks to the the Hardy-Littlewood inequality  the last term is $\lesssim \e^{-1}\mathcal H^1(\lbrace m=\tau\rbrace)$, and this concludes the proof of Proposition~\ref{p:adissiptrace}.
\end{proof}

\section{Compactness argument}\label{s:comp}

In this section we use the characterization of zero-energy states \cite{JOP} and a compactness argument to `initialize' our analysis of the previous sections:
if $\nu(\Omega)$ is small enough, then $\Omega$ must be close to a disk and $m$ close to a vortex.

\begin{lem}\label{l:comp}
For any $K,\e>0$, there exists $\delta=\delta(\e,K)>0$ with the following property.
If $\Omega$ is a $C^{1,1}$ simply connected domain with $\mathcal H^1(\partial\Omega)=2\pi$, $\sup_{\partial\Omega}|\kappa|\leq K$ which admits a map $m$ solving
\eqref{eq:eik} and \eqref{eq:kin} and its dissipation measure $\nu$ defined in \eqref{e:def_nu} satisfies $\nu(\Omega)\leq\delta$, then 
\begin{align}\label{e:dist_circle}
\dist(\partial\Omega,\partial B_1(x_0)) + \sup_{x\in\partial\Omega}\left| n_{\partial\Omega}(x)-\frac{x-x_0}{|x-x_0|}\right| 
\leq \e,
\end{align}
for some $x_0\in\R^2$, and there is $\alpha\in\lbrace \pm 1\rbrace$ such that
\begin{align}\label{e:trace_diffuse}
\int_{\Omega}\mathbf 1_{\dist(\cdot,\partial\Omega)\leq \frac {1}{2K}}|m-\alpha\tau_{\partial\Omega}\circ\pi_{\partial\Omega}|\,dx
\leq \e,
\end{align}
where $\pi_{\partial\Omega}(x)\in\partial\Omega$ is the nearest-point projection of $x$ onto $\partial\Omega$, well-defined for $\dist(x,\partial\Omega)\leq 1/(2K)$.
\end{lem}

\begin{proof}[Proof of Lemma~\ref{l:comp}]
Assume by contradiction that there exists a sequence of $C^2$ simply connected domains $\Omega_k$ such that $\mathcal H^1(\partial\Omega_k)=2\pi$, $\sup_{\partial\Omega_k}|\kappa|\leq K$ with maps $m_k\colon\Omega_k\to\mathbb S^1$ satisfying $\nabla\cdot m_k=0$ in $\Omega$, $m_k\cdot n_{\partial\Omega_k}=0$ on $\partial\Omega_k$ and $\nu_{m_k}(\Omega_k)=\delta_k\to 0$, such that
\begin{align*}
\inf_{x_0\in\R^2}\left\lbrace \dist(\partial\Omega_k,\partial B_1(x_0)) + \sup_{x\in\partial\Omega_k}\left| n_{\partial\Omega_k}(x)-\frac{x-x_0}{|x-x_0|}\right| \right \rbrace &\geq \e,
\end{align*}
or $\pi_{\partial\Omega_k}(x)$ is not well-defined (that is, not unique) for $\dist(x,\partial\Omega_k)\leq 1/(2K)$, or
\begin{align*}
\min_{\alpha\in\lbrace \pm 1\rbrace}\int_{\Omega_k}\mathbf 1_{\dist(\cdot,\partial\Omega_k)\leq \frac {1}{2K}}|m-\alpha\tau_{\partial\Omega_k}\circ\pi_{\partial\Omega_k}|\,dx
&
\geq \e.
\end{align*}

Let $\p_k\in C^2(\R/2\pi\Z;\R^2)$ be a counterclockwise arc-length parametrization of $\partial\Omega_k$. Up to a translation we may assume that $\int_{\R/2\pi\Z}\p_k(t)\, dt=0$.
Since $|\ddot\p_k|\leq K$ there exists a subsequence (which we don't relabel) such that $\p_k$ convergence in $C^1(\R/2\pi\Z;\R^2)$ to a curve $\p\in C^{1,1}(\R/2\pi\Z;\R^2)$ with $|\dot\p|=1$. The curve $\p$ can self-intersect, but not self-cross, so at a multiple point all tangents must be parallel.

Each domain $\Omega_k$ contains a disk of radius $\geq 1/K$ \cite[Proposition~2.1]{howard-treibergs}, so $\R^2\setminus \p(\R/2\pi\Z)$ has an open bounded simply connected component containing a disk of radius $\geq 1/K$, which we denote by $\Omega$.
 Since $\partial\Omega\subset \p(\R/2\pi\Z)$, the boundary $\partial\Omega$ is $C^1$ except at multiple points of the $C^1$ curve $\p$. 
We distinguish two types of singular points: a singular point $z\in\partial\Omega$ is of type I if there exists $\delta>0$ such that all connected components $\omega_\delta$ of $\Omega\cap B_\delta$ are such that $\partial\omega_\delta\cap\partial\Omega$ is $C^1$, and of type II otherwise. See Figure \ref{f:singularities}.
 The rest of the proof is divided in 4 steps.

\begin{figure}[h]
\centering
\def\svgwidth{0.45\columnwidth}
\input{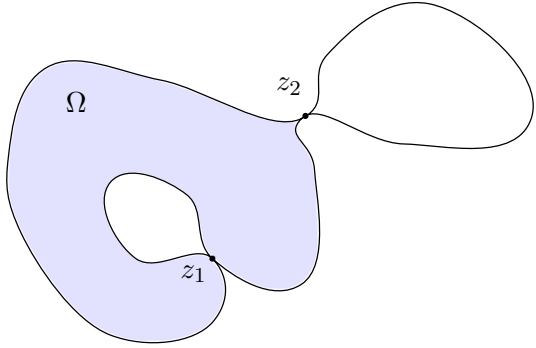}
\caption{\small The point $z_1$ is a type I singularity, while $z_2$ is a type II singularity.}
\label{f:singularities}
\end{figure}

\medskip
\noindent{\bf Step 1.} Singular points of type II are isolated in $\partial\Omega$.

\medskip

Let $z_0\in\partial\Omega$ a singular point. In particular it is a multiple point: $\p^{-1}(\lbrace z_0\rbrace)$ contains strictly more than one element. 
If $t_1\in\p^{-1}(\lbrace z_0\rbrace)$, then   $\p(t)\neq z_0$ for $|t-t_1|<2\pi/K$ because $|\dot\p|=1$ and  $|\ddot\p|\leq K$. This implies that $\p^{-1}(\lbrace z_0\rbrace)$ is finite, $\p^{-1}(\lbrace z_0\rbrace)=\lbrace t_1,\ldots,t_N\rbrace$ for some $N\geq 2$. Since $\p$ cannot self-cross, we have $\dot\p(t_j)\in\lbrace \pm\tau\rbrace$ for all $j=1,\ldots,N$ and one fixed $\tau=\dot\p(t_1)\in\mathbb S^1$.

We consider, for $\delta>0$ small enough, the open set $\p^{-1}(B_\delta(z_0))$. From a Taylor expansion around each $t_j$, we deduce the existence of $\eta=\eta(K)>0$ such that for $j=1,\ldots,N$, the subset $\p^{-1}(B_\delta(z_0))\cap (t_j-\eta,t_j+\eta)$ is an open interval $I_\delta^j$. For small enough $\delta>0$, the open set $\p^{-1}(B_\delta(z_0))$ is exactly the union of these $N$ open intervals. 
Maybe I would give for granted this proof, but since it is already there, we can also keep it. We prove this by contradiction: otherwise, there would exist a sequence $\delta\to 0$ and $t_\delta$ such that $|\p(t_\delta)-z_0|<\delta$ but $t_\delta\notin I_j^\delta$ for any $j=1,\ldots,N$. 
Extracting a subsequence $t_\delta\to t_*$, we must have $\p(t_*)=z_0$, so $t_*\in\lbrace t_1,\ldots,t_N\rbrace$, 
and therefore $t_\delta\in (t_j-\eta,t_j+\eta)$ for some $j\in\lbrace 1,\ldots,N\rbrace$ and all small enough $\delta$. 
By the above property of $\eta$ this implies $t_\delta\in I_\delta^j$ and gives a contradiction, so $\p^{-1}(B_\delta(z_0))$ is indeed the union of the $N$ intervals $I_\delta^j$.

 We choose coordinates $(x,y)$ in which $z_0=0$ and $\tau=e_1$. 
Then for small enough $\delta>0$ we have 
\begin{align*}
\p(\R/2\pi\Z)\cap B_\delta =\bigcup_{j=1}^N \lbrace y=f_j(x) \rbrace \cap B_\delta,
\end{align*}
for some $C^1$ functions $f_1\leq \cdots\leq f_N$ such that $f_j(0)=f_j'(0)=0$ for $j=1,\ldots,N$. 
We have $z_0=0\in\partial\Omega$, and any $(x_0,y_0)\in B_{\delta}\cap\Omega$ must satisfy
$y_0\notin\lbrace f_1(x_0),\ldots,f_N(x_0)\rbrace$.
Moreover, if $f_j(x_0)<y_0<f_{j+1}(x_0)$ for some $j\in\lbrace 1,\ldots,N-1\rbrace$, 
we deduce that $f_j(x)<f_{j+1}(x)$ for all $x\in (0,x_0)$
since $\Omega$ is connected.
Therefore, possibly taking a smaller value of $\delta$, the connected components of 
$\Omega\cap B_\delta$ are among the sets
\begin{align*}
&\lbrace y<f_1(x)\rbrace\cap B_\delta, \\
& \lbrace y> f_N(x)\rbrace\cap B_\delta,\\
&\lbrace f_j(x)<y<f_{j+1}(x)\rbrace\cap \lbrace x>0\rbrace \cap B_\delta,
\\
&\lbrace f_j(x)<y<f_{j+1}(x)\rbrace\cap \lbrace x<0\rbrace \cap B_\delta.
\end{align*}
Note that the singular point $z_0=0$ is of type I if and only if the two last types of connected components do not arise. Moreover, if $z_0$ is of type II, this description of $\Omega\cap B_\delta$ shows that $\partial\Omega\cap B_\delta$ contains no other singular points of type II. This proves Step 1.

\medskip
\noindent{\bf Step 2.} There exists $m\colon\Omega\to \mathbb S^1$ such that $\nabla\cdot m=0$ in $\Omega$, $\nu(\Omega)=0$, with a strong $L^1$ trace $m_{\lfloor\partial\Omega}$ satisfying $m_{\lfloor\partial\Omega}\cdot n_{\partial\Omega}=0$ a.e. on $\partial\Omega$. 
More precisely, this makes sense in any $C^1$ portion of $\partial\Omega$, singular points of type II are negligible by Step 1, and around a singular point of type I, $\Omega$ is, in adapted coordinates, locally of the form $\lbrace y<f(x)\rbrace\cup \lbrace y>\tilde f(x)\rbrace$ for some $C^1$ functions $f\leq \tilde f$ with $f(0)=\tilde f(0)=f'(0)=\tilde f'(0)=0$, and the trace $m_{\partial\Omega}$ may differ from one side to another, but both traces satisfy $m\cdot n_{\partial\Omega}=0$ a.e. for any choice of unit normal $n_{\partial\Omega}$.

\medskip

For any $z\in \Omega$ and $B_r(z)\subset\Omega$, the sequence $m_{k\lfloor D_r(z)}$ has bounded entropy production and is therefore compact in $L^1(D_r(z))$ \cite{dkmo01}, so $m_{k\lfloor\Omega}$ is compact in $L^1_{\mathrm{loc}}(\Omega)$. After extracting a subsequence converging in $L^1_{loc}$ and a.e., its limit   $m\colon\Omega\to\mathbb S^1$ satisfies $\nabla\cdot m=0$ in $\Omega$ and $\nabla\cdot\Phi(m)=0$ for any entropy $\Phi$ (see Appendix~\ref{a:ent}), so  $\nu(\Omega)=0$.
This last property implies that $m$ has an $L^1$ trace along any $C^1$ portion of $\partial\Omega$, and along both sides of any portion of $\partial\Omega$ around a singular point of type I  (see e.g. \cite[\S~3.2]{JOP}). 
It remains to prove that this trace satisfies $m_{\lfloor\partial\Omega}\cdot n_{\partial\Omega}=0$ a.e. on $\partial\Omega$. 

To that end consider first a $C^1$ point $z_0\in\partial\Omega$, and a disk  $B_{2r}(z_0)$ such that $\partial\Omega\cap B_{2r}(z_0)=\p(I)$ for some open interval $I$. Possibly choosing a smaller $r$ and adapted coordinates $(x,y)$ in which $z_0=0$ and $\p(I)$ is close to horizontal, we write $\Omega\cap B_{2r}$ as a subgraph
\begin{align*}
\Omega\cap B_{2r}=\lbrace y<f(x)\rbrace \cap B_{2r},
\end{align*}
where $f(0)=f'(0)=0$ and $f$ is $C^1$.
 Since $\p_k\to\p$ in $C^1$, we can write $\p_k(I)$ as a graph $\lbrace y=f_k(x)\rbrace$ for some functions $f_k$ converging to $f$ in $C^1$, and define
 \begin{align*}
 \widetilde\Omega_k = \lbrace y<f_k(x) \rbrace \cap B_r.
 \end{align*}
 Then we have $\widetilde\Omega_k\subset\Omega_k$ and $B_r\cap\partial\widetilde\Omega_k \subset \partial\Omega_k$, so $m_{k}\cdot n_{\partial\widetilde\Omega_k}=0$ on $B_r\cap\partial\widetilde\Omega_k \subset \partial\Omega_k$, implying that $\nabla\cdot (m_k\mathbf 1_{\widetilde\Omega_k})=0$ in $B_r$. 
By dominated convergence we have 
  $m_k \mathbf 1_{\widetilde\Omega_k} \to m\mathbf 1_{\Omega\cap B_r}$ in $L^1(B_r)$. We deduce that $\nabla\cdot (m\mathbf 1_{\Omega})=0$ in $B_r$, which implies that $m_{\lfloor\partial\Omega}\cdot n_{\partial\Omega}=0$ a.e. on $\partial\Omega\cap B_r$. 
 This is valid around any $C^1$ point of $\partial\Omega$.
Around a singular point $z_0$ of type I, the same argument can be applied in both connected components of $\Omega\cap B_{2r}(z_0)$.
And singular points of type II are isolated by Step~1, so $m_{\lfloor\partial\Omega}\cdot n_{\partial\Omega}=0$ a.e. on $\partial\Omega$. 

\medskip
\noindent{\bf Step 3.} There are no singular points of type II.
\medskip

If $z_0\in\partial\Omega$ is a singular point of type II, then by the analysis in Step~1 we may choose coordinates $(x,y)$ in 
 which $z_0=0$ and  there exist $\delta>0$ and a connected component $\omega_\delta$ of $\Omega\cap B_\delta$ such that
\begin{align*}
\omega_\delta=\lbrace f_1(x)<y<f_2(x) \rbrace\cap \lbrace x>0\rbrace\cap B_\delta,
\end{align*}
where $f_1,f_2$ are $C^1$ functions such that $f_1(0)=f_2(0)=f_1'(0)=f_2'(0)=0$ and $f_1(x)<f_2(x)$ for $x>0$.
For all small $x>0$, the normal $N_{f_1,x}$ to the graph of $f_1$ at $x$ intersects the graph of $f_2$, at a point $(x',f_2(x'))$. There must be at least a value of $x$ at which 
the normal $N_{f_2,x'}$ to the graph of $f_2$ at $x'$ is not parallel to $N_{1,x}$: otherwise the distance from $(x,f_1(x))$ to the graph of $f_2$ would be a constant function of $x$, contradicting the fact that the two graphs intersect at $0$.
Therefore, considering $x''$ slightly larger or smaller than $x'$ (depending on the sign of the angle between the two normals $N_{f_1,x}$ and $N_{f_2,x'}$) 
we have that $N_{f_2,x''}$ intersects $N_{f_1,x}$ at a point $z_1\in\omega_\delta$.
The proof of   \cite[Theorem~1.2]{JOP} implies that, for   every point $z\in\partial\Omega$ such that $n_{\partial\Omega}(x)$ is defined and $[z_1,z]\setminus\lbrace z\rbrace$ is contained in $\Omega$, $n_{\partial\Omega}(z)$ must be equal to $(z-z_1)/|z-z_1|$. 
By a continuation argument, we deduce that the graphs of $f_1$ and $f_2$ are, until they meet, arcs of circles centered at $z_1$, but this contradicts the fact that they have the same tangent at their intersection point. This contradiction proves that 
$\partial\Omega$ contains no singular points of type II.

\medskip
\noindent{\bf Step 4.} Conclusion.

\medskip

If there are no singular points of type I,
 $\Omega$ is $C^1$ with bounded curvature  and \cite[Proposition~2.1]{howard-treibergs} ensures the existence of a tangent inscribed disk centered at a focal point: explicitly we have a disk $B_r(z_*)\subset\Omega$ and a tangency point   $z_b=\p(t_0)\in\partial B_r(z_*)\cap \partial\Omega$, such that for any $\e>0$ and $z_\e=z_* -\e(z_b-z_*)$, the function $t\mapsto |\p(t)-z_\e|$ doesn't have a local minimum at $t_0$. 
In fact it can be checked that the proof of  \cite[Proposition~2.1]{howard-treibergs} works even if $\Omega$ has singular points of type I. (One may also approximate $\Omega$ with $C^1$ domains, apply \cite[Proposition~2.1]{howard-treibergs} and pass to the limit.)
So we do have
 a disk $B_r(z_*)\subset\Omega$ and a tangency point   $z_b=\p(t_0)\in\partial B_r(z_*)\cap \partial\Omega$, such that for any $\e>0$ and $z_\e=z_* -\e(z_b-z_*)$, the function $t\mapsto |\p(t)-z_\e|$ doesn't have a local minimum at $t_0$. 
The derivative of that function cannot be nondecreasing near $t_0$, and this implies the existence of $t\neq t_0$ arbitrarily close to $t_0$ such that the normal line to $\partial\Omega$ at $\p(t)$ intersects $[z_\e,z_b]$. 
Denote by $z_1\in\Omega$ one such intersection point. 
Applying again the argument in the proof of \cite[Theorem~1.2]{JOP}, we have the following geometric property: for every point $z\in\partial\Omega$ such that  $[z_1,z]\setminus\lbrace z\rbrace$ is contained in $\Omega$, $n_{\partial\Omega}(z)$ must be equal to $(z-z_1)/|z-z_1|$. 
Let $I\subset\R$ be the connected component of $t_0$ in the open set of all $t\in\R$ such that  $[z_1,\p(t)]\setminus \lbrace\p(t)\rbrace$ is contained in $\Omega$. Thanks to the above, we deduce that $\p(I)$ is an arc of   circle  centered at $z_1$. Assume $I$ doesn't coincide with $\R$, it means that there is a half-line from $z_1$ which intersects $\partial\Omega$ for the first time tangentially, but this is impossible by the above geometric property. We conclude that $I=\R$ and $\partial\Omega=\p(\R/2\pi\Z)$ is a circle $\partial B_R(z_1)$. Since $\p$ has length $2\pi$ we must have $R=1$, hence $\Omega=B_1(z_1)$. Further, from \cite[Theorem~1.2]{JOP} there exists $\alpha\in\lbrace \pm 1\rbrace$ such that
\begin{align*}
m(x)=\alpha i\frac{x-z_1}{|x-z_1|}=\alpha \tau_{\partial\Omega}\circ\pi_{\partial\Omega}(x)\qquad\text{for a.e. }x\in\Omega.
\end{align*}
Therefore, the  convergence of $\p_k$ to $\p$ in $C^1(\R/2\pi\Z;\R^2)$ implies
\begin{align*}
\dist(\partial\Omega_k,\partial B_1(z_1)) + \sup_{x\in\partial\Omega_k}\left| n_{\partial\Omega_k}(x)-\frac{x-z_1}{|x-z_1|}\right|\longrightarrow 0,
\end{align*}
and also that $\pi_{\partial\Omega_k}(x)$ is well defined for $\dist(x,\partial\Omega_k)\leq 1/(2K)$ and large enough $k$.
Moreover by dominated convergence we have 
\begin{align*}
\int_{\Omega_k}\mathbf 1_{\dist(\cdot,\partial\Omega_k)\leq \frac {1}{2K}}|m_k-\alpha\tau_{\partial\Omega_k}\circ\pi_{\partial\Omega_k}|\,dx
\longrightarrow 0.
\end{align*}
This contradicts the assumptions on $(\Omega_k,m_k)$ and concludes the proof of Lemma~\ref{l:comp}.
\end{proof}

\section{Trace estimate}\label{s:trace}

The estimate \eqref{e:trace_diffuse} provided by the compactness argument is not enough to handle the trace term in Proposition~\ref{p:adissiptrace}. In this section we explain how to strengthen it to a quantitative trace estimate, using the Lagrangian representation introduced in \S~\ref{ss:lag}.

\begin{prop}\label{p:trace}
Let $\Omega\subset \R^2$ be a $C^{1,1}$ simply connected domain with $\mathcal H^1(\partial\Omega)=2\pi$, $\sup_{\partial\Omega}|\kappa|\leq K$. For any map $m$ solving \eqref{eq:eik} and \eqref{eq:kin}, there is $\alpha \in \{\pm 1\}$ for which
\begin{align}\label{e:trace}
\mathcal H^1( \{ x \in \partial \Omega : m(x) = - \alpha \tau (x) \} ) \le C \nu (\Omega),
\end{align}
where $C>0$ is a constant depending only on $K$.
\end{prop}

 We start by showing a preliminary lemma, which is a more precise version of \cite[Lemma~3.1]{marconi21ellipse} 
 (see also \cite[Lemma~22]{marconi21micromag}).
 As in \cite[Lemma~2.7]{marconi21ellipse} we denote by $\Gamma_g\subset\Gamma$ the full measure set of curves $\gamma$ such that for a.e. $t\in (t_\gamma^-,t_\gamma^+)$ we have that $\gamma_x(t)$ is a Lebesgue point of $m$ with $m(\gamma_x(t))\cdot e^{i\gamma_s(t)}>0$.

\begin{lem}\label{l:gain_dimension}
Let $r>0$, $\bar \gamma \in \Gamma_g$ and $\bar t \in (t^-_{\bar \gamma}, t^+_{\bar \gamma})$ be such that $B_r(\bar \gamma_x(\bar t)) \subset \Omega$ and denote by $(t^-_r, t^+_r)$ the connected component of $\gamma_x^{-1}(B_r(\bar \gamma_x(\bar t)))$ in $ (t^-_{\bar \gamma}, t^+_{\bar \gamma})$ containing $\bar t$.
Then there exists an absolute constant $c>0$ such that for every $\beta \in \big(  \mathrm{Osc}_{(t^-_r,t^+_r)}\gamma_s, \pi/4\big)$ at least one of the following holds:
\begin{enumerate}
\item $\nu(B_r(\bar \gamma_x(\bar t))) \ge c \beta^3 r$;
\item $\mathcal L^2(\{x \in B_r(\bar \gamma_x(\bar t)) : e^{i\bar \gamma_s(\bar t)}\cdot m(x) \ge - 2 \beta \}) \ge c \beta r^2 $.
\end{enumerate}
\end{lem}
\begin{proof}[Proof of Lemma~\ref{l:gain_dimension}]
We let $\bar x = \bar \gamma_x(\bar t)$, $\bar s=\bar\gamma_s(\bar t)$, and  $\mathcal C_{\bar\gamma}$ be the image curve  $\mathcal C_{\bar\gamma}=\bar\gamma_x((t_{r}^-,t_{r}^+))\subset\Omega$. For $x=\bar\gamma(t)\in\mathcal C_{\bar\gamma}$ we denote by $\tau_{\bar\gamma}(x)=\dot {\bar\gamma}_x(t)$ the unit tangent vector determined by the parametrization $\bar\gamma$. In particular, $\tau_{\bar\gamma}(\bar x)=e^{i\bar s}$.

For $\mathcal H^1$-a.e. $x \in  \mathcal C_{\bar \gamma} \cap B_{r/2}(\bar x)$ we have $m(x) \cdot \tau_{\bar\gamma} (x) \ge 0$, therefore recalling that $\beta \in    \big(  \mathrm{Osc}_{(t^-_r,t^+_r)}\gamma_s, \pi/4\big)$ one of the following holds:
\begin{align*}
&m(x) \cdot e^{is} e^{i\bar s} >0 \quad \forall s \in \left(\frac{5\beta}{4}, \frac{7\beta}4\right),\\
\text{or }\quad & m(x) \cdot e^{is} e^{i\bar s} >0 \quad \forall s \in \left(-\frac{7\beta}{4}, -\frac{5\beta}4\right).
\end{align*}
One of these two conditions must be satisfied for at least half the points in $\mathcal C_{\bar \gamma} \cap B_{r/2}(\bar x)$, and we assume without loss of generality that 
\begin{align}\label{eq:tracegammabar}
\mathcal H^1\left( \left\{ x \in \mathcal C_{\bar \gamma} \cap B_{r/2}(\bar x) :  m(x) \cdot e^{is} e^{i\bar s} >0 \ \forall s \in \left(\frac{5\beta}{4}, \frac{7\beta}4\right) \right\}   \right) \ge \frac{r}2.
\end{align}
We define
\begin{align*}
I(\gamma) & = \{ t \in (0,T): \gamma_x(t) \in  B_{r}(\bar x) \}, \\
 I'(\gamma)  & = \{ t\in I(\gamma) : \gamma_s(t) \in (\bar s+\beta, \bar s+2 \beta)\}, 
\end{align*}
and  $ T(\gamma) = \mathcal L^1(I'(\gamma))$.
We moreover consider
\begin{align}\label{eq:Ngamma}
 N(\gamma)= \#  \left\{ t \in (0,T) : \gamma_x(t) \in \mathcal C_{\bar \gamma} \cap B_{r/2}(\bar x), \gamma_s(t) \in \left( \bar s  +\frac{5\beta}{4}, \bar s + \frac{7\beta}4\right) \right\}.
\end{align}
This cardinal is finite for $\omega$-a.e. $\gamma\in\Gamma$ thanks to \cite[Proposition~3.3]{CHLM22},
and we denote by $  t_1(\gamma)< \ldots <   t_{  N(\gamma)}(\gamma)$ the elements of the above set.
We show that for every $\gamma \in \Gamma$ we have
\begin{equation}\label{e:gamma_fixed}
\frac{\mu_\gamma(I(\gamma))}\beta + \frac{T(\gamma)}r \ge \frac14  N(\gamma),
\end{equation}
where $\mu_\gamma=|D_t\gamma_s|\in\mathcal M (t_\gamma^-,t_\gamma^+)$ can be interpreted as the entropy dissipation along $\gamma$ thanks to Proposition~\ref{p:lagrangiandissipation}.
It follows from the characteristic equation \eqref{e:charac} that for every $i =1,\ldots  N(\gamma)$ there is a neighbourhood $I_i$ of $ t_i(\gamma)$ of size at least $r/2$ such that $I_i\subset I(\gamma)$ and at least one of the following holds:
\[
I_i \subset I'(\gamma) \qquad \mbox{or} \qquad \mu_\gamma(I_i) \ge \frac{\beta}4.
\]
The neighborhoods $I_i$ are not necessarily disjoint, but if $i=1,\ldots  N(\gamma)-1$ is such that $ t_{i+1}(\gamma)-  t_i(\gamma) < r$, then \eqref{e:charac} implies that  $[ t_i(\gamma),   t_{i+1}(\gamma)]\subset I(\gamma)$ and $\mu_\gamma([ t_i(\gamma), t_{i+1}(\gamma)]) \ge \beta/4$.
This establishes  \eqref{e:gamma_fixed}.

Next we integrate \eqref{e:gamma_fixed} with respect to $\omega$. From Proposition~\ref{p:lagrangiandissipation} we deduce
\[
\int_\Gamma \mu_\gamma(I(\gamma)) \, d \omega \le \nu (B_r(\bar x)) T,
\]
and from the Lagrangian property \eqref{E_repr_formula} we infer
\begin{align*}
\int_\Gamma T(\gamma) \, d \omega  &\le  T \mathcal L^3(\{ (x,s) \in B_r(\bar x) \times (\bar s+ \beta, \bar s+ 2 \beta) \colon m(x) \cdot e^{is} >0 \}) \\
& \le  T \beta \mathcal L^2(\{ x \in B_r(\bar x) \colon m(x) \cdot e^{i\bar s} >-2 \beta\}).
\end{align*}
Therefore integrating \eqref{e:gamma_fixed} we obtain
\begin{align}\label{eq:intNup}
\frac 1T \int_\Gamma N(\gamma) \, d\omega  \leq \frac{4}{\beta}\nu(B_r(\bar x))  + \frac{4\beta}{r} 
\mathcal L^2(\{ x \in B_r(\bar x) : m(x) \cdot e^{i\bar s} >-2\beta\}).
\end{align}
To estimate from below the left-hand side of \eqref{eq:intNup} we use its link with the Lagrangian flux across the curve $\mathcal C_{\bar \gamma}$. Specifically,
for any $f\in C_c^1((0,T)\times\Omega\times \R/2\pi\Z)$, we have
\begin{align*}
&\int f(t,x,s) \mathbf 1_{m(x)\cdot e^{is}>0}(i\tau_{\bar\gamma}(x)\cdot e^{is})\, ds\, d\mathcal H^1_{\lfloor \mathcal C_{\bar\gamma}}(x)\, dt 
=\int_\Gamma \langle F_\gamma,f\rangle\, d\omega(\gamma),
\end{align*}
where $F_\gamma$ is given by
\begin{align*}
\langle F_\gamma,f\rangle &=\sum_{t\in X^+}f(t,\gamma_x(t),\gamma_s(t^+)) -\sum_{t\in X^-}f(t,\gamma_x(t),\gamma_s(t^-)) \\
X^+&=\left\lbrace 
t\in (t_\gamma^-,t_\gamma^+)\colon \gamma_x(t)\in \mathcal C_{\bar\gamma},\; i\tau_{\bar\gamma}(\gamma_x(t))\cdot e^{i\gamma_s(t)}>0 
\right\rbrace  \\
X^-&=\left\lbrace t\in (t_\gamma^-,t_\gamma^+)\colon \gamma_x(t)\in \mathcal C_{\bar\gamma},\; i\tau_{\bar\gamma}(\gamma_x(t))\cdot e^{i\gamma_s(t^-)}<0 \right\rbrace.
\end{align*}
The set $X^+$ corresponds to intersection times of $\gamma$ with $\mathcal C_{\bar\gamma}$ where $\gamma$ exits $\mathcal C_{\bar\gamma}$ in direction of the normal $i\tau_{\bar\gamma}$, and the set $X^-$  to intersection times where $\gamma$ enters $\mathcal C_{\bar\gamma}$ in the opposite direction. Note that these two sets may not be disjoint since $\gamma$ could `bounce' on $\mathcal C_{\bar\gamma}$. 
The proof of this flux formula is similar to the proof of Lemma~\ref{l:projinit} for the boundary flux, and details are provided in \cite[Theorem~1.4]{CHLM22} in a very similar setting.
 Applying this flux formula to
\begin{align*}
f(t,x,s) \approx \mathbf 1_{t\in (0,T)}\mathbf 1_{x\in B_{r/2}(\bar x)}\mathbf 1_{s\in (\bar s+5\beta/4,\bar s+7\beta/4)},
\end{align*}
we see that there are no contributions from $X^-$ and obtain 
\begin{align*}
\int_\Gamma   N(\gamma) d\omega 
& =T\int \mathbf 1_{m(x)\cdot e^{is}>0}(i\tau(x)\cdot e^{is})\mathbf 1_{s\in [\bar s +\frac{5\beta}{4},\bar s +\frac{7\beta}{4}]} \,ds\, d\mathcal H^1_{\lfloor \mathcal C_{\bar\gamma} \cap B_{r/2}(\bar x)}(x).
\end{align*}
Using also \eqref{eq:tracegammabar} we deduce
\begin{align*}
\frac 1T \int_\Gamma   N(\gamma) d\omega 
&
 \geq    \sin\left(\frac{\beta}{4}\right)\frac{\beta}{2}
 \frac r2 \geq \frac{1}{8\pi} \beta^2r.
\end{align*}
Combining this with \eqref{eq:intNup} we get
\[
\frac{1}{32\pi}\beta^2 r  \le  \frac 1 \beta \nu(B_r(\bar x))  +  \frac{\beta}r \mathcal L^2(\{ x \in B_r(\bar x) : m(x) \cdot e^{i\bar s} >-2\beta\}) .
\]
This implies the statement of Lemma~\ref{l:gain_dimension}.
\end{proof}

With Lemma~\ref{l:gain_dimension} at hand, we turn to the proof of Proposition~\ref{p:trace}.

\begin{proof}[Proof of Proposition~\ref{p:trace}]
It is sufficient to prove the statement for $\nu (\Omega) < \delta$ for some small $\delta$. 
We choose $\alpha$ satisfying \eqref{e:trace_diffuse} and we prove that  \eqref{e:trace_diffuse} implies \eqref{e:trace}, provided $\delta$ is sufficiently small.
Assume without loss of generality that $\alpha = -1$ and let us consider the set of curves 
\begin{align*}
G=\Big\{\gamma \in \Gamma \colon 
&
0< t^-_\gamma < T-1,\  m(\gamma_x(t^-_\gamma)) = \tau(\gamma_x(\tau_\gamma^-)), \\
& \mbox{and }e^{i\gamma_s(t^-_\gamma)} \cdot  e^{i\frac\pi 4}\tau(\gamma_x(\tau_\gamma^-)) \ge \cos\left(\frac\pi{16}\right)
\Big\}.
\end{align*}
By Lemma \ref{l:projinit} we have 
\begin{align}\label{e:omegaG}
\omega(G) \ge \cos \left(\frac\pi4+ \frac\pi{16}\right) \frac{\pi (T-1)}{16} \mathcal H^1( \{ x \in \partial \Omega : m(x) =  \tau (x) \} ).
\end{align}

\noindent {\bf Claim}. If $\nu(\Omega)$ is sufficiently small, then   $\omega$-a.e. $\gamma\in G$ satisfies $\mu_\gamma((t^-_\gamma, t^+_\gamma))\ge \frac1{32}$. 

\medskip

The Claim implies the statement since 
\[
T \nu(\Omega) \ge \int_G \mu_\gamma((t^-_\gamma, t^+_\gamma)) d\omega \ge \frac{\omega(G)}{32}
\]
and eventually \eqref{e:trace} follows by \eqref{e:omegaG}.

It remains to prove the Claim. 
Let $\e >0$ small to be chosen later and assume $\nu(\Omega)<\delta$ where $\delta = \min\{\delta',\delta(\e,K)\}$ where $\delta(\e,K)$ is provided by Lemma \ref{l:comp} and $\delta'>0$ is chosen later.
Assume by contradiction that there is $\bar \gamma \in G\cap \Gamma_g$ such that $\mu_{\bar\gamma}((t^-_{\bar \gamma}, t^+_{\bar \gamma}))< \frac{1}{32}$. The constraints \eqref{e:dist_circle} and \eqref{e:extreme_time} imply that $t^+_{\bar \gamma}- t^-_{\bar \gamma} \ge \frac1K$. 
Moreover setting $\bar t = t^-_{\bar \gamma} + 4r$, and $\bar x = \bar \gamma_x(\bar t)$, we have that if we assume $r\leq 1/(100 K)$, then $B_{r}(\bar x) \subset \{x \in \Omega: \mathrm{dist}(x,\partial \Omega) < \frac1{2K}\}$. 
We can therefore apply Lemma \ref{l:gain_dimension} with $\beta = \frac{1}{32}$
and get that one of the following holds true:
\begin{enumerate}
\item $\nu(\Omega) \ge \nu(B_r(\bar x)) \ge c \beta^3 r$;
\item the set $A=\{x \in B_r(\bar x) \colon e^{i\bar \gamma_s(\bar t)}\cdot m(x) \ge - 2\beta \}$ satisfies 
\begin{align}\label{e:lowerboundA}
\mathcal L^2(A) \ge c \beta r^2.
\end{align}
\end{enumerate}
The first case is incompatible with $\nu(\Omega)<\delta'$, provided $\delta' <c \beta^3r$. Therefore we take $\delta' = \frac{c}2\beta^3 r$, with $r \leq 1/(100 K)$ to be fixed later. 

Let us then consider the second case: we are going to show that \eqref{e:lowerboundA} is contradicts \eqref{e:trace_diffuse} for $\e$ sufficiently small.
First we observe that every $x \in B_r(\bar x)$ satisfies
\[
\mathrm{dist}(\pi_{\partial \Omega}(x), \bar \gamma_x(t^-_{\bar \gamma})) \le \tilde c(1+K)r
\]
for some absolute constant $\tilde c>0$, and therefore
\begin{align*}
e^{i\bar \gamma_s(\bar t)}\cdot \tau(\pi_{\partial \Omega}(x)) 
&\ge  e^{i\bar \gamma_s(\bar t)}\cdot \tau ( \bar \gamma_x(t^-_{\bar \gamma})) - 
|  \tau ( \bar \gamma_x(t^-_{\bar \gamma})) -  \tau(\pi_{\partial \Omega}(x))| \\
&\ge  e^{i\bar \gamma_s(\bar t)}\cdot \tau ( \bar \gamma_x(t^-_{\bar \gamma}))  - \tilde cK(1+K)r.
\end{align*} 
Moreover, using that $\gamma\in G$ and $\mu_{\bar\gamma}((t^-_{\bar \gamma}, t^+_{\bar \gamma}))< \frac{1}{32}$ we infer
\begin{align*}
e^{i\bar \gamma_s(\bar t)}\cdot \tau(\pi_{\partial \Omega}(x))
&\ge  e^{i\bar \gamma_s(t^-_{\bar \gamma})}\cdot \tau ( \bar \gamma_x(t^-_{\bar \gamma}))  - |e^{i\bar \gamma_s(t^-_{\bar \gamma})} -e^{i\bar \gamma_s(\bar t)}| -  \tilde cK(1+K)r \\
&\ge  \cos \left( \frac\pi{16} + \frac\pi4\right) - \frac1{32} - \tilde cK(1+K)r \\
&\ge  \frac14 -  \tilde cK(1+K)r,
\end{align*}
for all  $x \in B_r(\bar x)$.
Let us choose 
\[r = \min\left\{ \frac1{100K}, \frac1{8  \tilde cK(1+K)} \right\},
\]
so that by the above
\begin{align*}
e^{i\bar \gamma_s(\bar t)}\cdot \tau(\pi_{\partial \Omega}(x)) \ge \frac18\qquad\forall x\in B_r(\bar x).
\end{align*}
We deduce in particular
\begin{align*}
 \int_{\Omega} \mathbf{1}_{\dist(\cdot,\partial\Omega)\leq \frac {1}{2K}}|m+\tau\circ\pi_{\partial\Omega}|\,dx
&\ge  \int_{A} (m+\tau\circ\pi_{\partial\Omega}) \cdot e^{i\bar \gamma_s(\bar t)} \, dx 
\\
&\ge \left( \frac18 - 2\beta \right) \mathcal L^2(A) 
= \frac1{16} \mathcal L^2(A) 
\\
&\geq \frac{c}{16}\beta r^2.
\end{align*}
The last inequality follows from \eqref{e:lowerboundA}.
Choosing $ \e = c\beta r^2 /32$ contradicts 
 \eqref{e:trace_diffuse} and concludes the proof of the Claim and of Proposition~\ref{p:trace}.
\end{proof}

\section{Proof of the main results}
We collect the results from the previous sections to prove Theorem~\ref{t:main} and Corollary~\ref{c:estOmega}.
We moreover prove Corollary~\ref{c:estm} and Proposition \ref{p:sharp}.

\subsection{Proof of Theorem~\ref{t:main} and Corollary~\ref{c:estOmega}}\label{s:proofmain}

Let $m$ solve \eqref{eq:eik} and \eqref{eq:kin}.
Without loss of generality, we may assume that the constant $\alpha$ provided by the trace estimate Proposition~\ref{p:trace} is equal to $-1$, hence
\begin{align*}
\mathcal H^1( \{ x \in \partial \Omega : m(x) = \tau (x) \} ) \le C \nu (\Omega).
\end{align*}
Lemma~\ref{l:comp} ensures that, if $\nu(\Omega)$ is small enough, then 
$\Omega=E(\eta)$, with $\eta =\min(\eta_0,\eta_*)$, $\eta_0$ as in Proposition~\ref{p:est_a_circle} and $\eta_*$ as in Proposition~\ref{p:adissiptrace}. 
Gathering the results of both said Propositions together with the above trace estimate, we obtain Theorem~\ref{t:main2} and \eqref{eq:estOmega}, which imply Theorem~\ref{t:main} and Corollary~\ref{c:estOmega} as explained in the introduction.

\subsection{Proof of Corollary~\ref{c:estm}}\label{s:proofestm}

In this section we prove \eqref{eq:estm}, which implies Corollary~\ref{c:estm}. We rely on a div-curl argument involving the entropies $\Sigma_1,\Sigma_2\colon\mathbb S^1\to\R^2$ introduced in \cite{JK00}, given by
\begin{align*}
\Sigma_1(m)&=\frac 43 (m_2^3,m_1^3),\\
\Sigma_2(m)&= \frac 23\left( -m_1^3-3m_1m_2^2,m_2^3 + 3 m_2 m_1^2\right).
\end{align*}

\begin{lem}\label{l:divcurl}
For any $m_1,m_2\colon\Omega\to\mathbb S^1$ with strong $L^1$ traces on $\partial\Omega$ we have
\begin{align*}
\| m_1-m_2\|^3_{L^4(\Omega)} \leq c_0 \|\nabla\cdot\Sigma(m_1)-\nabla\cdot\Sigma(m_2)\|_{\mathcal M(\Omega)} + c_0 K \|m_1-m_2\|_{L^1(\partial\Omega)},
\end{align*}
where $\Sigma=(\Sigma_1,\Sigma_2)$  and $\nabla\cdot\Sigma(m)=(\nabla\cdot\Sigma_1(m),\nabla\cdot\Sigma_2(m))$. The constant $c_0=c_0(\Omega)$ depends on the norm of the Sobolev embedding $W_0^{1,4}(\Omega)\subset L^\infty(\Omega)$, and on $K=\max_{\partial\Omega} |\kappa |$.
\end{lem}

\begin{proof}[Proof of Lemma~\ref{l:divcurl}]
The proof is inspired by \cite{golse10}.
Let $\chi\in C_c^\infty(\Omega)$ such that $|\chi|\leq 1$, and apply the div-curl estimate of \cite[Lemma~4.2]{GL} to the vector fields
\begin{align*}
E=\chi(\Sigma_1(m_1)-\Sigma_1(m_2)),\qquad B=\chi(\Sigma_2(m_1)-\Sigma_2(m_2)).
\end{align*}
for $p=4$.
This yields
\begin{align*}
\int E\wedge B 
&\lesssim \|\chi(\Sigma(m_1)-\Sigma(m_2))\|_{L^4(\Omega)}
\|\nabla\cdot (\chi(\Sigma(m_1)-\Sigma(m_2))\|_{W^{-1,4/3}(\Omega)} \\
&\lesssim \| \chi (m_1-m_2)\|_{L^4(\Omega)}
\|\nabla\cdot (\chi(\Sigma(m_1)-\Sigma(m_2))\|_{W^{-1,4/3}(\Omega)}.
\end{align*}
we used $|\nabla\Sigma|\lesssim 1$ for the last inequality. Moreover, thanks to \cite[Lemma~7 and (92)]{LP18}  we have
\begin{align*}
E\wedge B =\chi^2\det(\Sigma(m_1)-\Sigma(m_2))\gtrsim \chi^2 |m_1-m_2|^4\geq \chi^4|m_1-m_2|^4.
\end{align*}
Therefore we deduce from the previous inequality that
\begin{align}\label{eq:divcurl1}
\|\chi(m_1-m_2)\|_{L^4}^3 \lesssim 
\|\nabla\cdot (\chi(\Sigma(m_1)-\Sigma(m_2))\|_{W^{-1,4/3}(\Omega)}.
\end{align}
Next we compute, for any $\zeta\in W_0^{1,4}(\Omega)\subset C_0(\Omega)$,
\begin{align*}
&\left\langle \nabla\cdot (\chi(\Sigma(m_1)-\Sigma(m_2)),\zeta\right\rangle
\\
& =
\int\zeta\nabla\chi\cdot (\Sigma(m_1)-\Sigma(m_2))
+ \langle \nabla\cdot\Sigma(m_1)-\nabla\cdot \Sigma(m_2),\chi\zeta\rangle \\
&\lesssim \|\zeta\|_{\infty} \|\nabla\chi(m_1-m_2)\|_{L^1} 
+ \|\zeta\|_{\infty}
\|\nabla\cdot\Sigma(m_1)-\nabla\cdot\Sigma(m_2)\|_{\mathcal M(\Omega)} \\
&\lesssim c_0(\Omega) \|\nabla\zeta\|_{L^4}
\left(\|\nabla \chi(m_1-m_2)\|_{L^1} + 
\|\nabla\cdot\Sigma(m_1)-\nabla\cdot\Sigma(m_2)\|_{\mathcal M(\Omega)} 
\right), 
\end{align*}
where $c_0(\Omega)$ is the norm of the Sobolev embedding $W_0^{1,4}(\Omega)\subset L^\infty(\Omega)$. By definition of $W^{-1,4/3}=(W_0^{1,4})'$ this implies 
\begin{align*}
&\|\nabla\cdot (\chi(\Sigma(m_1)-\Sigma(m_2))\|_{W^{-1,4/3}(\Omega)}
\\
&\lesssim 
c_0(\Omega) 
\left(\|\nabla\chi(m_1-m_2)\|_{L^1} + 
\|\nabla\cdot\Sigma(m_1)-\nabla\cdot\Sigma(m_2)\|_{\mathcal M(\Omega)} 
\right).
\end{align*}
Plugging this into \eqref{eq:divcurl1} we obtain
\begin{align*}
\|\chi(m_1-m_2)\|_{L^4}^3 \lesssim 
c_0(\Omega) 
\left(\|\nabla\chi(m_1-m_2)\|_{L^1} + 
\|\nabla\cdot\Sigma(m_1)-\nabla\cdot\Sigma(m_2)\|_{\mathcal M(\Omega)} 
\right).
\end{align*}
Finally, choosing $\chi=\chi_\e$ such that $\chi_\e(x)=1$ for $\dist(x,\partial\Omega)>\e$ and $|\nabla\chi|\lesssim K/\e$ for $\e\to 0$ and using the trace property of $m_1,m_2$ we obtain the result.
\end{proof}

Applying Lemma~\ref{l:divcurl} to $m_1=m$ and $m_2=m_*=i(x-x_*)/|x-x_*|$ we deduce, using  Theorem~\ref{t:main2} to estimate $\|m-m_*\|_{L^1(\partial\Omega)}$,
\begin{align*}
\|m-m_*\|_{L^4(\Omega)}^3 \leq C\, \|\nabla\cdot\Sigma(m)\|_{\mathcal M(\Omega)} + C\, \nu(\Omega)^{\frac 12}.
\end{align*}
Moreover since $\Sigma_1,\Sigma_2$ are entropies (see Appendix~\ref{a:ent}), the first term in the right-hand side is controlled by $\nu(\Omega)$, and we directly deduce \eqref{eq:estm}.\qed

\subsection{Proof of Proposition \ref{p:sharp}}\label{s:sharp}

Given $N\ge 3$, we define $\Omega_N$ as the convex hull of the union of the disks $D_{1/2}(e^{2ik\pi/N}/2)$, $k=0,\ldots, N-1$, rescaled by a factor $1+\mathcal O(1/N^2)$ in order to have perimeter $2\pi$. In other words, $\Omega_N$ is obtained from the regular $N$-gon replacing sharp corners by arcs of circles, see Figure~\ref{f:OmegaN}.
The set $\Omega_N$ is $C^{1,1}$ with $\sup_{\partial\Omega_N}|\kappa|\leq 2$.

\begin{figure}[h]
\centering
\def\svgwidth{0.55\columnwidth}
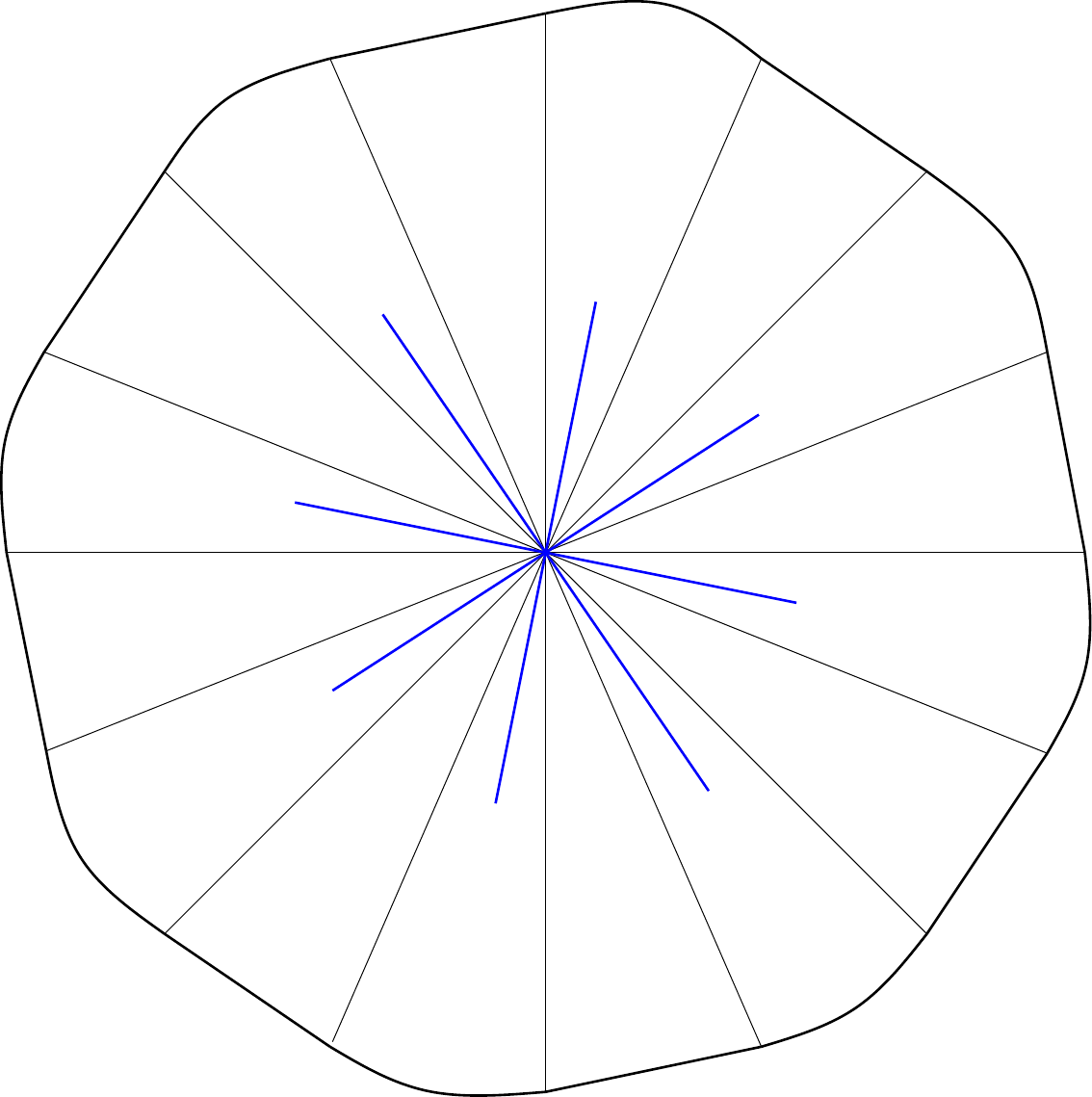
\caption{\small The figure represents $\Omega_N$ with $N=8$. Its boundary is composed by 8 segments and 8 circular arcs. The blue segments form the set where $\nabla \dist(\cdot, \partial \Omega)$ is discontinuous.}\label{f:Nagon}
\label{f:OmegaN}
\end{figure}

The unit normal $n_{\partial\Omega_N}$ is equal to the disk's unit normal $x/|x|$ at the $2N$ points of polar angle $e^{i\pi \ell/N}$ ($\ell=0,\ldots, 2N-1$), and it differs from it by $\sim 1/N$ in $2N$ boundary arcs of length $\sim 1/N$ away from those points. Therefore we have
\begin{align}\label{eq:normalOmegaN}
\int_{\partial\Omega_N}
\left| n_{\partial\Omega_N}(x)-\frac{x}{|x|}\right|^2\,d\mathcal H^1(x)\sim \frac {1}{N^2}.
\end{align}
One particular solution $m_N$ of \eqref{eq:eik} and \eqref{eq:kin} in $\Omega_N$ is given by
\begin{align*}
m_N(x)=i\nabla\dist(\cdot,\partial\Omega_N).
\end{align*}
This map $m_N$ is $BV$, its jump set $J_N$ is the union of $N$ segments,
\begin{align*}
J_N =\bigcup_{k=0}^{N-1} [0,x_k],\qquad x_k=\left(\frac 12 + \mathcal O\left(\frac 1{N^2}\right)\right)e^{2ik\pi/N}/2,
\end{align*}
with jump amplitude 
\begin{align*}
|m_N^+-m_N^-| \lesssim \frac{1}{N}\qquad\text{on }J_N.
\end{align*}
Replacing the sharp jump along $J_N$ with a well-chosen smooth transition at scale $\e$, one obtains maps $\tilde m_{\e,N}\to m_N$ as $\e\to 0$, with
\begin{align}\label{eq:upperboundOmegaN}
\limsup_{\e\to 0} F_\e(\tilde m_{\e,N};\Omega_N) \lesssim \int_{J_N}|m_N^+-m_N^-|^3\,d\mathcal H^1 \lesssim N \cdot \frac{1}{N^3} =\frac{1}{N^2}.
\end{align}
Details of such construction can be found e.g. in \cite{poliakovsky07,contidel} for the Aviles-Giga functional $E_\e^{AG}$, which is enough to obtain an upper bound on $F_\e$. 
For the functional $E_\e^{RS}$ a similar construction is performed in \cite{RS01},  and the methods in \cite{poliakovsky07} apply for $E_\e^{ARS}$.
(Note that for our explicit map $m_{N}$  the technical details of such construction can be significantly simplified because the jump set $J_N$ is particularly simple and stays away from the boundary, and $m_N$ is smooth outside of it.) Combining \eqref{eq:normalOmegaN}, \eqref{eq:upperboundOmegaN} and \eqref{eq:main} we obtain Proposition~\ref{p:sharp}.\qed

\begin{rem}
We cannot prove that the minimizers $m_{\e,N}$ of $E_\e^{AG}(\cdot; \Omega_N)$ and $E_\e^{ARS}(\cdot;\Omega_N)$ converge to $m_{N}$ as $\e \to 0$, but from the proof above we have that $m_{N}$ and the (possibly different and not unique) limit of $m_{\e,N}$ go to 0 with the same order as $N \to \infty$. 
\end{rem}

\begin{appendices}

\section{Entropy productions, compactness and kinetic formulation}\label{a:ent}

The kinetic formulation \eqref{eq:kin} is intimately linked to the notion of entropy, also borrowed from conservation laws, and introduced in \cite{dkmo01} for the eikonal equation. A smooth map $\Phi\colon\mathbb S^1\to\R^2$ is an entropy for the eikonal equation \eqref{eq:eik} if and only if it preserves the divergence-free quality of smooth solutions :
\begin{align*}
(|m|=1\text{ and }\nabla\cdot m=0 )\quad
\Rightarrow
\quad
\nabla\cdot\Phi(m)=0,
\end{align*}
for any open $\Omega\subset\R^2$ and smooth $m\colon\Omega\to\R^2$. 
Direct calculation shows that this is equivalent to the existence of a smooth function $\lambda_\Phi\colon\R/2\pi\Z\to\R$ such that
\begin{align*}
\frac{d}{d\theta}\Phi(e^{i\theta})=\lambda_\Phi(\theta)ie^{i\theta}\qquad\forall\theta\in\R.
\end{align*}
To any $f\in C^\infty(\R/2\pi\Z)$ one may associate an entropy $\Phi_f$ given by
\begin{align}\label{eq:Phif}
\Phi_f(z)=\int_{\R/2\pi\Z} f(s)\mathbf 1_{z\cdot e^{is}>0}\, ds\qquad\forall z\in\mathbb S^1,
\end{align}
and the kinetic formulation \eqref{eq:kin} is equivalent to
\begin{align}\label{eq:dvPhif}
\langle \nabla\cdot \Phi_f(m),\zeta \rangle =-\langle \sigma, 
 f'(s)\zeta(x) \rangle\qquad\forall \zeta\in C_c^\infty(\Omega),f\in C^\infty(\R/2\pi\Z).
\end{align}
An entropy $\Phi$, whenever extended to $\R^2$ by setting $\widehat\Phi(re^{i\theta})=\eta(r)\Phi(e^{i\theta})$ for some fixed real-valued cut-off function $\eta\in C_c^\infty(0,\infty)$ with $\eta(1)=1$, satisfies (see e.g. \cite{JOP,DeI})
\begin{align*}
\nabla\cdot\widehat\Phi(m)=\Psi(m)\cdot\nabla (1-|m|^2) +\alpha(m)\nabla\cdot m\qquad\forall m\in H^1(\Omega;\R^2),
\end{align*}
where $\Psi\colon\R^2\to\R^2$ and $\alpha\colon\R^2\to\R$ are such that
\begin{align*}
\|\nabla\Psi\|_{C^1(\R^2)} + \|\nabla\alpha\|_{C^1(\R^2)} \leq c \|\lambda_\Phi\|_{C^1(\R/2\pi\Z)},
\end{align*}
for some constant $c>0$ depending only on the cut-off function $\eta$. Applying this to $m'_\e=(m^1_\e,m^2_\e)$  for some sequence $m_\e\in H^1(\Omega;\R^3)$ with $F_\e(m_\e)\leq C$, we find
\begin{align}\label{eq:dvPhimeps}
\nabla\cdot\widehat\Phi(m'_\e)&=\nabla\cdot \left[\Psi(m'_\e)(1-|m'_\e|^2) -\alpha(m'_\e)H_\e\right]
\nonumber \\
&
\quad
-(1-|m'_\e|^2)\nabla\cdot\Psi(m'_\e)-H_\e\cdot\nabla [\alpha(m'_\e)]\qquad\text{ in }\mathcal D'(\Omega),
\end{align}
where $H_\e\colon\R^2\to\R^2$ is the curl-free vector field such that $\nabla\cdot H_\e =-\nabla\cdot(\mathbf 1_\Omega m'_\e)$.
 Boundedness of the energy $F_\e(m_\e;\Omega)$ implies that the first line in the right-hand side of \eqref{eq:dvPhimeps} tends to $0$ in $H^{-1}(\Omega)$, while the second line is bounded in $L^1(\Omega)$. 
One can then argue exactly as in \cite{dkmo01}, to deduce that $\nabla\cdot\widehat\Phi(m'_\e)$ is precompact in $H^{-1}(\Omega)$ and that $m'_\e$ is precompact in $L^2(\Omega)$.
This gives the precompactness of $m_\e$ since $m_\e^3\to 0$ in $L^2(\Omega)$.
Moreover taking the limit $\e\to 0$ in \eqref{eq:dvPhimeps} along a converging subsequence $m_\e\to m$, one infers
\begin{align*}
\langle \dv \Phi(m),\zeta\rangle \lesssim \|\zeta\|_\infty \|\lambda_\Phi\|_{C^1} \liminf_{\e\to 0}F_\e(m_\e;\Omega)\qquad\forall \zeta\in C_c^\infty(\Omega).
\end{align*}
Using the arguments of \cite[\S~3.1]{GL} (see \cite[Appendix~B]{LPfacto} for more details), this estimate provides the existence of $\sigma\in\mathcal M(\Omega\times \R/2\pi\Z)$ satisfying \eqref{eq:kin}.

\section{On the sharp lower bound for $E_\e^{ARS}$}\label{a:ARS}

The analysis recalled in  Appendix~\ref{a:ent} provides an energy lower bound 
in terms of the kinetic dissipation measure of the limit map. 
In the case of $E_\e^{ARS}$ \eqref{eq:ARS} from \cite{ARS}, these arguments can be refined to obtain a sharp lower bound: for any $m=\lim m_\e$ we have
\begin{align}\label{eq:ARSsharp}
\frac 12 \nu(\Omega)\leq \liminf_{\e\to 0} E_\e^{ARS}(m_\e;\Omega),
\end{align}
where $\nu$ is the minimal kinetic dissipation measure associated to $m$ as defined in \eqref{e:def_nu}. Moreover this lower bound is sharp if $m\in BV(\Omega;\R^2)$, in the sense of $\Gamma$-convergence: there exists $m_\e\to m$ such that 
\begin{align}\label{eq:ARSup}
\limsup_{\e\to 0} E_\e^{ARS}(m_\e;\Omega) \leq \frac 12 \nu(\Omega).
\end{align}
The sharp lower bound \eqref{eq:ARSsharp} is contained in \cite{ARS}, but not explicitly stated, so we briefly recall here why it is valid. The key step is \cite[Lemma~2.2]{ARS}, which ensures the existence of $\tilde m_\e \in W^{1,p}(\Omega;\mathbb S^1)$ for $1\leq p<2$, such that $\tilde m_\e\to m$ and
\begin{align}\label{eq:lemmaARS}
\int_\Omega |\nabla \tilde m_\e| \cdot |\tilde H_\e| \, dx \leq E_\e^{ARS}(m_\e;\Omega) +o(1).
\end{align}
Here $\tilde H_\e\colon\R^2\to\R^2$ is the curl-free vector field such that $\nabla\cdot \tilde H_\e =-\nabla\cdot (\mathbf 1_\Omega\tilde m_\e)$. The gain provided by this lemma is that $\tilde m_\e$ takes values into $\mathbb S^1$, so one can directly compute entropy productions (without using an extension $\widehat\Phi$ as in the previous section). Specifically, for an entropy $\Phi$ we have
\begin{align*}
\nabla\cdot \Phi(\tilde m_\e)&=\lambda_\Phi(\tilde m_\e)\nabla\cdot\tilde m_\e\\
& =- \nabla\cdot \left[ \lambda_\Phi(\tilde m_\e)\tilde H_\e\right]
+ \lambda_\Phi'(\tilde m_\e) H_\e \cdot \nabla\tilde m_\e\qquad\text{in }\mathcal D'(\Omega).
\end{align*}
Here $\lambda_\Phi(e^{i\theta})=ie^{i\theta}\cdot (d/d\theta)\Phi(e^{i\theta})$ as in Appendix~\ref{a:ent}. As in Appendix~\ref{a:ent} this implies
\begin{align*}
|\nabla\cdot \Phi(m)|(\Omega) \leq \|\lambda_\Phi'\|_\infty \liminf_{\e\to 0} E_\e^{ARS}(m_\e;\Omega).
\end{align*}
This is also the argument in Step~1 of the proof of \cite[Theorem~1]{ARS}. 
A natural refinement of that argument (see e.g. the proof of  \cite[Proposition~2]{llp22}) leads to
\begin{align*}
 \left(\bigvee_{\Phi\in S}|\nabla\cdot \Phi(m)|\right)(\Omega)
 \leq \liminf_{n\to\infty} E_\e^{ARS}(m_\e;\Omega),
\end{align*}
where $S$ is any class of entropies $\Phi$ with $\|\lambda_\Phi'\|_\infty\leq 1$, and $\bigvee$ denotes the lowest upper bound measure of a family of measures \cite[Definition~1.68]{ambrosio}. Applying this to entropies $\Phi_f$ as in \eqref{eq:Phif}, which satisfy 
$\lambda_{\Phi_f}(\theta)=f(\theta+\pi/2)+f(\theta-\pi/2)$, we deduce 
\begin{align*}
\left(\bigvee_{|f'|\leq 1/2}|\nabla\cdot \Phi_f(m)|\right)(\Omega) \leq \liminf_{n\to\infty} E_\e^{ARS}(m_\e;\Omega).
\end{align*}
Recalling \eqref{eq:dvPhif} and \eqref{e:def_nu}, we see that the left-hand side is equal to $(1/2)\nu(\Omega)$, which proves \eqref{eq:ARSsharp}.

For a $BV$ map $m$, we let $J_m$ denote its jump set and $m^\pm$ the traces of $m$ along $J_m$.
Then, the calculations in \cite[Corollary~3.4]{marconi21ellipse}  imply that we have
\begin{align*}
\frac 12 \nu(\Omega) &= \int_{J_m} c(|m_+-m_-|)\, d\mathcal H^1, \\
\text{where }c(2\sin X)&=\begin{cases}
2 \left|\sin X - X \cos X \right| & \text{if }0\leq X\leq \pi/4,\\
2\left| (X-\pi/2)\cos X - \sin X +\sqrt{2} \right| & \text{if }\pi/4 \leq X\leq \pi/2.
\end{cases}
\end{align*}
This is exactly the expression of the lower bound in \cite[Theorem~1]{ARS}.
Moreover, that lower bound is shown to be  optimal  in \cite[Theorem~2]{ARS}, in the sense that the energy cost $A(X)=c(2\sin X)$ corresponds to the asymptotic energy per unit-length of an ideal wall transition between to limit values $m^\pm$ with $|m^+-m^-|=2\sin X$. 
This implies the $\Gamma$-upper bound \eqref{eq:ARSup} using e.g. the techniques in \cite{poliakovsky13}.

\begin{rem}\label{r:ARS}
A closer look at Step~1 in the proof of \cite[Theorem~1]{ARS} reveals that only entropies of the form $\Phi_{f^\sigma}$ are used to obtain the lower bound \eqref{eq:ARSsharp}, where $2f^\sigma(s)=g(s-\sigma)$ for any $\sigma\in\R$ and $g$ is $\pi$-periodic with $g(s)=\pi/4-|s-\pi/4|$ for $s\in [-\pi/4,3\pi/4]$. 
This should come as no surprise, since, as a 
 consequence of the disintegration of $\sigma_{\mathrm{min}}$ in \cite[Corollary~3.4]{marconi21ellipse}, it can be checked that the identity
\begin{align*}
\frac{1}{2}\nu(\Omega)&=\bigvee_{|f'|\leq 1/2} |\nabla\cdot\Phi_f(m)| 
=\bigvee_{\sigma\in\R}|\nabla\cdot\Phi_{f^\sigma}(m)|,
\end{align*}
is valid for any $m$ satisfying the kinetic formulation \eqref{eq:kin}.
\end{rem}

\section{Quantitative alternative to the compactness argument under a restrictive trace assumption}\label{a:alt}

In this appendix we prove that, if the integral of $a$ is small enough, then $\Omega$ is close enough to a disk. 
This provides a quantitative proof of the estimate \eqref{e:dist_circle} obtained via the compactness argument of Lemma~\ref{l:comp}. 
We are however not able to prove \eqref{e:trace_diffuse} without a compactness argument, so that this only leads to a quantitative proof of Theorem~\ref{t:main2} under the additional trace assumption that $m \cdot \tau$ is constant on $\partial\Omega$.  

\begin{prop}\label{p:soft}
Let $\Omega$ as in Theorem \ref{t:main}.
For any $\eta>0$ there is $c=c(\eta,K)>0$ such that
 if $\int_{E(\eta)^3}a\, d (\mathcal H^1)^{\otimes 3} \le c$, then $E(\eta) = \partial \Omega$.
\end{prop}

The main ingredient to prove Proposition~\ref{p:soft} is the following lower bound on $a$ at one boundary triple, if $\Omega$ fails to be close enough to $\partial D$.

\begin{lem}\label{l:pointwise_estimate}
For any $\eta>0$ there is a constant $a_0=a_0(\eta,K)>0$ such that, if $E(\eta)\neq \partial\Omega$ then there exists $\hat x\in E(\eta)^3$ with
\[
a(\hat x) \ge a_0.
\]
\end{lem}

\begin{proof}
We choose coordinates in which $x_0=0$ and consider $\eta\leq\e_0/K$ for some small absolute constant $\e_0>0$ to be adjusted during the proof: for larger values of $\eta$ we can then simply take $a_0=a_0(\e_0/K,K)$. 

We assume that $E(\eta)\neq \partial\Omega$ and prove the existence of $\hat x$ satisfying $a(\hat x)\geq a_0$ in several steps. During the proof we denote by $c_0$ a generic small constant that depends only on $\eta$ and $K$. We are going to construct a triple $\hat x =(x_1,x_2,x_3)\in E_*^3$ and three directions $e^{i\alpha_1},e^{i\alpha_2},e^{i\alpha_3}$ that can be used in the definition of $a(\hat x)$ to show that $a(\hat x)\geq a_0$. We divide this construction in 5 steps.

\medskip

\noindent\textbf{Step 1.} There exists $y\in E(\eta)$ such that
\begin{align*}
\tau(y)\cdot \frac{y}{|y|} > \frac{\eta R}{2\pi}.
\end{align*}

\medskip

Pick a tangency point $\bar x=\p(\bar s)\in\partial\Omega\cap\partial B_R$, and let $(s_1,s_2)\subset\R$ denote the largest interval containing $\bar s$ and such that $\p((s_1,s_2))\subset E(\eta)$. Since $E(\eta)\neq\partial\Omega$ we know that $s_2-s_1 <2\pi$. Moreover, by Lemma~\ref{l:continuation} if $\e_0\leq 1/4$ we must have
\begin{align*}
|\p(s_1)|=|\p(s_2)|=(1+\eta)R.
\end{align*}
Consider the function $\psi(s)=|\p(s)| - R=\dist(\p(s),\partial B_R)$ as in Lemma~\ref{l:geom1}. We have $\psi(s_1)=\psi(s_2)=\eta R$ and $\psi(\bar s)=0$, so there must exist $s_*\in (\bar s,s_2)$ such that $\psi'(s_*)\geq \eta R/(s_2-s_1) > \eta R/(2\pi)$. Setting $y=\p(s_*)$ and recalling the expression \eqref{e:der_psi} of $\psi'$, we have $\psi'(s_*)=\tau(y)\cdot y/|y|> \eta R/(2\pi)$, proving Step~1.

\medskip

\noindent\textbf{Step 2.} For all angles $|\theta | \leq \eta R/(8\pi^2 K)$, the ray $\lbrace t e^{i\theta}y\rbrace_{t>0}$ does not contain any tangency point $x\in \partial\Omega\cap\partial B_R$. 

\medskip 

Recall from \eqref{e:der_psi} that $|\psi''|\leq 2K$. As $\psi'(s_*)>\eta R/(2\pi)$ this implies $\psi'(s)>0$ for all $s$ such that $|s-s_*|\leq \eta R/(4\pi K)$, hence $\p(s)$ is not a tangency point.
Further, as $\p$ is 1-Lipschitz and $|\p(s_*)|\leq (1+\eta)R$ we have $|\p(s)|\leq (1+2\eta)R$ for all $s$ such that $|s-s_*|\leq \eta R$, which by Lemma~\ref{l:continuation} implies $\p((s_*-\eta R,s_*+\eta R))\subset E(2\eta )$, provided $\e_0\leq 1/8$. 
Since $K\geq 1/R\geq 1$ we deduce that whenever $|s-s_*| \leq \eta R/(4\pi K)$ we have $\p(s)\in E(2\eta )$ and the ray $\lbrace t \p(s)\rbrace_{t>0}$ does not contain any tangency point. 

Let $\theta\colon\R \to \R$ be the $C^1$ function such that $\theta(s_*)=0$ and
\begin{align*}
\frac{\p(s)}{|\p(s)|}=e^{i\theta(s)}\frac{y}{|y|}.
\end{align*}
It satisfies
\begin{align*}
\theta'(s)=\frac{1}{|\p(s)|}\frac{i\p(s)}{|\p(s)|}\cdot \dot\p(s).
\end{align*}
If $|s-s_*|\leq \eta R/(4\pi K)$ we have $i\p(s)\cdot \dot\p(s)\geq 0$  and 
\begin{align*}
\frac{i\p(s)}{|\p(s)|}\cdot \dot\p(s) & =\sqrt{1-\left(\frac{\p(s)}{|\p(s)|}\cdot \dot\p(s)\right)^2}
 \geq \sqrt{1-8K\eta R}.
\end{align*}
The last inequality follows from Lemma~\ref{l:geom1} and the fact that $\p(s)\in E(2\eta )$. Since $R\leq 1$, provided $\e_0\leq 1/16$ we deduce $\dot\p(s)\cdot i\p(s)/|\p(s)|\geq 1/2$, and therefore
\begin{align*}
\theta'(s)\geq \frac{1}{2|\p(s)|}\geq \frac{1}{2\pi},
\end{align*}
using that $|\p(s)|\leq \pi$ as a consequence of $\mathcal H^1(\partial\Omega)=2\pi$. We deduce that 
\begin{align*}
\left[-\frac{\eta R}{8\pi^2K},\frac{\eta R}{8\pi^2 K}\right]\subset
\theta\left(
\left[s_* -\frac{\eta R}{4\pi K},s_* + \frac{\eta R}{4\pi K}\right]
\right),
\end{align*}
Therefore, if $|\theta|\leq \eta R/(8\pi^2K)$ then there exists $s$ such that $|s-s_*|\leq \eta R/(4\pi K)$ and the ray $\lbrace te^{i\theta}y\rbrace_{t>0}$ coincides with the ray $\lbrace t\p(s)\rbrace_{t>0}$, which does not contain any tangency point. This proves Step~2.

\medskip

\noindent\textbf{Step 3.} There exists a tangency point $\tilde x\in \partial\Omega\cap \partial B_R$ such that
\begin{align*}
\frac{\tilde x}{|\tilde x|}=e^{i\tilde\theta_0} \frac{y}{|y|} \qquad\text{ with }
\frac{\eta R}{8\pi^2 K} \leq  \tilde \theta_0 \leq  \pi-\frac{\eta R}{8\pi^2 K}.
\end{align*}

\medskip

By maximality of the inscribed disk $B_R\subset \Omega$, the tangency points cannot be all contained in an arc of angle less than $\pi$, so there must be at least one tangency point $\tilde x\in \partial\Omega\cap\partial B_R$ such that $\tilde x/|\tilde x|=e^{i\tilde\theta_0}y/|y|$ for some $\tilde\theta_0\in [-\eta R/(8\pi^2 K),\pi-\eta R/8\pi^2 K]$. Thanks to Step~2, it must satisfy also $\tilde\theta_0 \geq \eta R/(8\pi^2 K)$, proving Step~3.

\medskip

\noindent\textbf{Step 4.} There are constants $c_1,c_2,c_3>0$ depending only on $K$,  with the following property. For any tangency point $\tilde x\in\partial\Omega\cap\partial B_R$, any $t\in (0,1/2)$ and $z=t\tilde x$, and any $\delta\in (0,1/(8K))$, there exist
 $x_1,x_2\in E(2\delta^2)$ such that
\begin{align*}
&\frac{x_1-z}{|x_1-z|}=e^{i\theta_1}\frac{\tilde x}{|\tilde x|}\quad\text{for some }\theta_1 \in (-c_2\delta,-c_1\delta t),\\
&\frac{x_2-z}{|x_2-z|}=e^{i\theta_2}\frac{\tilde x}{|\tilde x|}\quad\text{for some }\theta_2 \in (c_1\delta t,c_2\delta  ),\\
& \tau(x_1)\cdot \frac{x_1-z}{|x_1-z|}\leq -c_3 \delta t,\qquad
 \tau(x_2)\cdot \frac{x_2-z}{|x_2-z|}\geq c_3 \delta t.
\end{align*}
These will be used in Step~5 as illustrated by Figure~\ref{fig}.

\medskip

\begin{figure}[h]
\centering
\def\svgwidth{0.55\columnwidth}
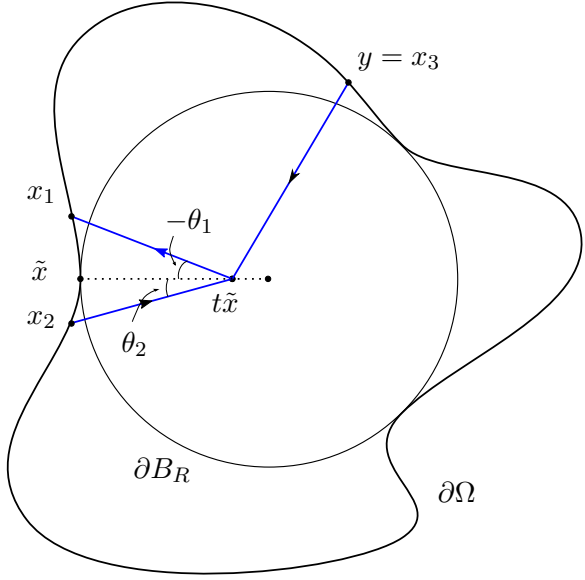
\caption{\small The blue arrows denote the directions $e^{i\alpha_1},e^{i\alpha_2},e^{i\alpha_3}$ in Step~5. The idea is that the two directions $e^{i\alpha_1}$ and $e^{i\alpha_2}$ are almost opposite and $e^{i\alpha_3}$ is not close to $e^{i\alpha_1}$ and $e^{i\alpha_2}$  and belongs to the longest of the two intervals with endpoints at $e^{i\alpha_1}$ and $e^{i\alpha_2}$.}
\label{fig}
\end{figure}

Write $\tilde x=\p(\tilde s)$ for some $\tilde s\in\R$. The map $\p$ is $1$-Lipschitz and $|\p(\tilde s)|=R$, so by Lemma~\ref{l:continuation} we have $\p(s)\in E_*$ for $|s-\tilde s|\leq \delta \leq 1/(4K)$.  

Consider the $C^1$ function $\hat\theta\colon\R\to\R$ such that $\hat\theta(\tilde s)=0$ and
\begin{align*}
\frac{\p(s)}{|\p(s)|}=e^{i\hat\theta(s)}\frac{\tilde x}{|\tilde x|}.
\end{align*}
As in Step~2 we have
\begin{align*}
\hat\theta'(s)=\frac{1}{|\p(s)|}\frac{i\p(s)}{|\p(s)|}\cdot \dot\p(s).
\end{align*}
Since $|\p(\tilde s)|=R$ and $\p$ is 1-Lipschitz, Lemma~\ref{l:geom1} implies
\begin{align*}
\left( \frac{\p(s)}{|\p(s)|}\cdot \dot\p(s) \right)^2
&\leq  4 K |s-\tilde s| \leq \frac 12\qquad\text{for }|s-\tilde s|\leq\delta\leq \frac{1}{8K}.
\end{align*}
We deduce as in Step~2 that $\hat\theta'(s)\geq 1/(2\pi)$ for $|s-\tilde s|\leq \delta$. Using also $|\p|\geq R$ we therefore have
\begin{align*}
\frac{1}{2\pi}\leq \hat\theta'(s)\leq \frac{1}{R}\qquad\text{for }|s-\tilde s|\leq\delta.
\end{align*}
On the other hand, setting
\begin{align*}
\varphi(s)=|\p(s)-z|^2= \left| \frac{|\p(s)|}{R} e^{i\hat\theta(s)}\tilde x-t\tilde x\right|^2,
\end{align*}
we have
\begin{align*}
\varphi(s)-\varphi(\tilde s)&
=\left| \frac{|\p(s)|}{R} \tilde x-t\tilde x -\frac{|\p(s)|}{R} (1-e^{i\hat\theta(s)})\tilde x\right|^2 -(R-Rt)^2\\
&
=(|\p(s)|-Rt)^2-(R-Rt)^2 +|\p(s)|^2 |1-e^{i\hat\theta(s)}|^2 \\
&\quad - 2|\p(s)|(|\p(s)|-Rt)(1-\cos\hat\theta(s))
\\
&
=(|\p(s)|-Rt)^2-(R-Rt)^2 +2Rt|\p(s)|(1-\cos\hat\theta(s)) \\
&\geq 2t R^2 (1-\cos\hat\theta(s)),
\end{align*}
where the last inequality follows from $|\p|\geq R$ and $t\leq 1$. 
As $\hat\theta(\tilde s)=0$ and $|\hat\theta'|\leq 1/R$, for $|s-\tilde s|\leq \delta$ we have 
$|\hat\theta(s)|\leq \delta/R\leq\delta K\leq 1/8$, and this implies $1-\cos\hat\theta(s)\geq \hat\theta(s)^2/4$, so
\begin{align*}
\varphi(s)-\varphi(\tilde s)
&\geq \frac t2 R^2 \hat\theta(s)^2\qquad\text{for }|s-\tilde s|\leq \delta.
\end{align*}
Using moreover that $\hat\theta'\geq 1/(2\pi)$ we deduce
\begin{align*}
\varphi(s)-\varphi(\tilde s) \geq \frac{tR^2}{8\pi^2}(s-\tilde s)^2\qquad\text{for }|s-\tilde s|\leq \delta.
\end{align*}
Hence there  exist $s_1\in (\tilde s-\delta,\tilde s)$, $s_2\in (\tilde s,\tilde s +\delta)$ such that
\begin{align*}
\varphi'(s_1)&=-\frac{1}{\delta}(\varphi(\tilde s-\delta)-\varphi(\tilde s)) \leq -\frac{R^2}{8\pi^2}\delta t,\\
\varphi'(s_2)&=\frac{1}{\delta}(\varphi(\tilde s+\delta)-\varphi(\tilde s)) \geq \frac{R^2}{8\pi^2}\delta t.
\end{align*}
Since $\varphi'(s)=2\tau(\p(s))\cdot (\p(s)-z)$ and
 $ |\p(s)-z|\leq \pi$, 
setting $x_1=\p(s_1)$, $x_2=\p(s_2)$
we obtain
\begin{align*}
\tau(x_1)\cdot\frac{x_1-z}{|x_1-z|}&=\frac{\varphi'(s_1)}{2|\p(s_1)-z|} \leq -c_3 \delta t\\
\tau(x_2)\cdot\frac{x_2-z}{|x_2-z|}&=\frac{\varphi'(s_2)}{2|\p(s_2)-z|} \geq c_3 \delta t, \qquad c_3 =\frac{R^2}{16\pi^3 }.
\end{align*}
This proves the last assertion of Step~4. Moreover, since $\p'(\tilde s)=0$, $\delta\in (0,1/(8K)$ and $|\ddot\p|\leq K$, the points $x_1,x_2$ lie outside of the disk of radius $1/K$ tangent to $D_R(x_0)$ at $\tilde x$, and since $R\geq 1/K$ we infer that they are at distance at most $2K\delta^2$ from $D_R(x_0)$, and thanks to Lemma~\ref{l:continuation} they belong to $E(2K\delta^2/R)\subset E(2\delta^2)$.

It remains to show that 
\begin{align*}
\frac{x_j-z}{|x_j-z|}=e^{i\theta_j}\frac{\tilde x}{|\tilde x|} \quad\text{for }j=1,2, 
\quad -c_2\delta<\theta_1<-c_1\delta t,\; c_1 \delta t <\theta_2 < c_2\delta.
\end{align*}
By definition of $\hat\theta$ we know that
\begin{align*}
\frac{x_j}{|x_j|}=e^{i\hat\theta(s_j)}\frac{\tilde x}{|\tilde x|},
\end{align*}
so we relate $\hat\theta(s_j)$ to $\theta_j$ and estimate $\hat\theta(s_j)$. To do the first, consider, for any fixed $s\in\R$, the  $C^1$ function $\alpha_s\colon [0,1/2]\to\R$ such that $\alpha_s(0)=\hat\theta(s)$ and
\begin{align*}
\frac{\p(s)-t\tilde x}{|\p(s)-t\tilde x|}=e^{i\alpha_s(t)}\frac{\tilde x}{|\tilde x|}.
\end{align*}
That way, we have  $\hat\theta(s_j)=\alpha_{s_j}(0)$ and can choose $\theta_j=\alpha_{s_j}(t)$. Moreover we have
\begin{align*}
\alpha_s'(t)=-\frac{i(\p(s)-t \tilde x)}{|\p(s)-t\tilde x|^2}\cdot \tilde x = \frac{\p(s)\cdot (i\tilde x)}{|\p(s)-t\tilde x|^2}=\frac{R|\p(s)|}{|\p(s)-t\tilde x|^2}\sin \hat\theta(s).
\end{align*}
Note that since $\hat\theta(\tilde s)=0$ and $1/(2\pi)\leq\hat\theta'\leq 1/R$   we have
\begin{align*}
0<\mathrm{sign}(s-\tilde s)\hat\theta(s)\leq \frac{\delta}{R}\leq \delta K\leq \frac 18
\quad\text{for }|s-\tilde s|\leq \delta.
\end{align*}
In particular, using $|\p|\leq \pi$, $|\p-t\tilde x|\geq R/2$ and $|\sin\hat\theta|\leq|\hat\theta|$, we deduce
\begin{align*}
0<\mathrm{sign}(s-\tilde s)\alpha_s'(t)\leq \frac{4\pi}{R^2}\delta
\leq \frac{\pi}{2}K
\qquad\text{for }|s-\tilde s|\leq \delta,
\end{align*}
hence, recalling $\theta_j-\hat\theta(s_j)=\int_0^t\alpha_{s_j}'$, we infer
\begin{align*}
-(1+\pi/2)K\delta\leq \theta_2\leq\hat\theta(s_1)<0<\hat\theta(s_2)\leq\theta_2 \leq (1+\pi/2)K\delta.
\end{align*}
The proof of Step~4 will be complete once we show that $|\hat\theta(s_j)|\geq c_1\delta t$ for $j=1,2$.
Because $\varphi'(\tilde s)=0$ and $|\varphi''|\leq 2K+2 \leq 4K$, we must have
\begin{align*}
|s_j-\tilde s|\geq\frac{|\varphi'(s_j)|}{4K}\geq 2\pi c_1 t,\quad c_1=\frac{R^2}{32\pi^2K}.
\end{align*}
Combining this with $\hat\theta'\geq 1/(2\pi)$ on $[s_1,s_2]$ we deduce that $|\hat\theta(s_j)|\geq c_1\delta t$ and conclude the proof of Step~4.

\medskip

\textbf{Step 5.} We choose $t\in (0,1/2)$ and $\delta\in (0,1/(8K))$ such that $2\delta^2\leq\eta$ and, for the tangency point $\tilde x\in\partial\Omega\cap\partial B_R$ obtained in Step~3 and $z=t\tilde x$, letting $x_1,x_2\in\partial\Omega$ provided by Step~4, and $x_3=y$ provided by Step~1, the three concurring lines from $x_1,x_2,x_3$ through $z$ can be used to show $a(x_1,x_2,x_3)\geq a_0$. See Figure~\ref{fig}.

\medskip

Let $\alpha_1,\alpha_2,\alpha_3\in\R/2\pi\Z$ such that
\begin{align*}
\frac{x_1-z}{|x_1-z|}=e^{i\alpha_1},\quad 
\frac{x_2-z}{|x_2-z|}=-e^{i\alpha_2},\quad
\frac{y-z}{|y-z|}=-e^{i\alpha_3}.
\end{align*}
By definition, the three lines $x_j+e^{i\alpha_j}\R$ are concurrent in $z\in B_{R/2}$.
Moreover by Step~4 we have
\begin{align*}
 \tau(x_j)\cdot e^{i\alpha_j} \leq -c_3\delta t\qquad\text{for }j=1,2.
\end{align*}
The function $t\mapsto \tau(y)\cdot (y-t\tilde x)/|y-t\tilde x|$ is $2$-Lipschitz on $[0,1/2]$ since $|\tilde x|=R$ and $|y-t\tilde x|\geq R/2$ for $t\in [0,1/2]$. Since $\tau(y)\cdot y/|y| > \eta R/(2\pi)$ by Step~1, choosing $t\in (0,\eta R/(8\pi))$ ensures
\begin{align*}
\tau(y)\cdot e^{i\alpha_3} =-\tau(y)\cdot \frac{y-t\tilde x}{|y-t\tilde x|} < -\frac{\eta R}{4\pi}.
\end{align*}
Recall from Step~3 that we have
\begin{align*}
\frac{\tilde x}{|\tilde x|}=e^{i\tilde\theta_0} \frac{y}{|y|} \qquad\text{ with }
\frac{\eta R}{8\pi^2 K} \leq  \tilde \theta_0 \leq  \pi-\frac{\eta R}{8\pi^2 K}.
\end{align*}
The $C^1$ function $\tilde \theta\colon [0,1/2]\to\R$ such that $\tilde\theta(0)=\tilde\theta_0$ and 
\begin{align*}
\frac{\tilde x}{|\tilde x|}=e^{i\tilde\theta(t)} \frac{y-t\tilde x}{|y-t\tilde x|}, 
\end{align*}
satisfies, arguing as in previous steps, $|\tilde\theta'|\leq 2$, so choosing 
\begin{align*}
t=\frac{\eta R}{32\pi^2K} \in (0,\eta R/(8\pi)),
\end{align*}
ensures
\begin{align*}
\frac{\tilde x}{|\tilde x|}=e^{i\tilde\theta_t} \frac{y-z}{|y-z|}, 
\qquad\text{ with }
\frac{\eta R}{16\pi^2 K} \leq  \tilde \theta_t \leq  \pi -\frac{\eta R}{16\pi^2 K}
\end{align*}
From this identity, the definitions of $\theta_1,\theta_2$ in Step~4, and the definitions of $\alpha_1,\alpha_2,\alpha_3$, we obtain
\begin{align*}
\frac{\tilde x}{|\tilde x|}=e^{i(\alpha_3 +\tilde\theta_t -\pi)}=e^{i(\alpha_2-\theta_2-\pi)}=e^{i(\alpha_1-\theta_1)}.
\end{align*}
So we have, recalling from Step~4 the inequalities satisfied  by $\theta_1,\theta_2$,
\begin{align*}
e^{i\alpha_2}&=e^{i(\pi+\theta_2-\theta_1)}e^{i\alpha_1},\qquad &
\pi+\theta_2-\theta_1 \in [\pi +2c_1\delta t,\pi+2c_2\delta],\\
e^{i\alpha_3}& =e^{i(\pi-\tilde\theta_t-\theta_1)}e^{i\alpha_1},\qquad
&\pi-\tilde\theta_t-\theta_1 \in [\eta R/(16\pi^2K) ,\pi-\eta R/(16\pi^2K) + c_2\delta].
\end{align*}
Choosing
\begin{align*}
\delta =\min\left( \frac{\eta R}{32c_2\pi^2 K},\frac{1}{2c_2},\frac{1}{8K},\sqrt{\frac\eta 2}\right),
\end{align*}
this implies that the shortest interval in $\R/2\pi\Z$ containing $\alpha_1,\alpha_2,\alpha_3$  is of length
\begin{align*}
l(\alpha_1,\alpha_2,\alpha_3) \geq \pi +\min\left(2c_1 \delta t, \frac{\eta R}{32 \pi^2 K}\right).
\end{align*}
Letting
\begin{align*}
a_0=\min\left(
2c_1 \delta t, \frac{\eta R}{32 \pi^2 K}, c_3\delta t \right),
\end{align*}
and gathering the above,
we conclude that $a(x_1,x_2,x_3)\geq a_0$.
\end{proof}

The proof of Proposition~\ref{p:soft} will be a combination of Lemma~\ref{l:pointwise_estimate} and of the fact, proven in Lemma~\ref{l:a_Lip}, that $a$ is Lipschitz.


\begin{proof}[Proof of Proposition~\ref{p:soft}]
We assume that $E(\eta)\ne \partial \Omega$, and prove that $\int_{E(\eta)^3} a d (\mathcal H^1)^{\otimes 3} \ge c$ for some $c=c(\eta,K)>0$.
As $E(\eta/2)\subset E(\eta)\varsubsetneq \partial\Omega$, applying Lemma~\ref{l:pointwise_estimate} we find
 $\hat x=(x_1,x_2,x_3) \in E(\eta/2)^3$ such that $a(\hat x)\geq a_0$, where $a_0=a_0(\eta,K)>0$. 
 Let $x_k=\p(\bar s_k)$ for $k=1,2,3$. Thanks to Lemma~\ref{l:continuation} and the Lipschitz quality of $a$ (Lemma~\ref{l:a_Lip}) we may choose $\delta=\delta(\eta,K)>0$ such that
 \begin{align*}
a\geq \frac{a_0}{2}\qquad\text{on } \prod_{k=1}^3 C([\bar s_k -\delta,\bar s_k+\delta])\subset E(\eta)^3.
 \end{align*}
 This implies
 \begin{align*}
 \int_{E(\eta)^3} a \,d (\mathcal H^1)^{\otimes 3}
\geq \delta^3\frac{a_0}{2},
 \end{align*}
 concluding the proof of Proposition~\ref{p:soft}.
\end{proof}

\end{appendices}

\section*{Acknowledgements} 
The authors wish to thank Andrew Lorent and Guanying Peng for many interesting discussions, which led among other things to the proof of Lemma~\ref{l:divcurl}. 
E.M. acknowledges the support received from the SNF Grant 182565 and the European Union's Horizon 2020 research and innovation program under the Marie Sk\l odowska-Curie grant No. 101025032.

\bibliographystyle{acm}
\bibliography{aviles_giga}

\end{document}